\RequirePackage{fix-cm}
\documentclass[smallextended]{svjour3}       
\smartqed  
\usepackage{graphicx}
\usepackage{amsmath}
\usepackage{amssymb} 
\usepackage[colorlinks=true, linkcolor=blue]{hyperref}

\hypersetup{
pdftitle={Capturing bifurcations in Hamiltonian boundary value problems numerically},
pdfauthor={Christian Offen},
pdfsubject={Bifurcation Theory, symplectic geometry, geometric integration boundary value problem},
pdfkeywords={bifurcation, geometric integration, boundary value problems, symplectic geometry, Lagrangian, Hamiltonian systems, integrable systems, periodic orbits, invariant tori, pitchfork, saddle-node, singularity theory, catastrophe theory},
bookmarksdepth=10,
}

\renewcommand{\d}{\mathrm{d}}

\newcommand{\D}{\mathrm{D}}

\newcommand{\e}{\epsilon}

\newcommand{\Z}{\mathbb{Z}}

\newcommand{\R}{\mathbb{R}}

\spdefaulttheorem{prop}{Proposition}{\upshape}{\itshape}

 \journalname{Numerical Algorithms}

\begin{document}

\newsavebox{\smlmata}
\savebox{\smlmata}{${\scriptstyle\overline q(0)=\left(\protect\begin{smallmatrix}0.2\\ 0.1\protect\end{smallmatrix}\right) = \overline q(5)}$}
\newsavebox{\smlmatb}
\savebox{\smlmatb}{${\scriptstyle q =\left(\protect\begin{smallmatrix} -1&2\\3&1\protect\end{smallmatrix}\right) \overline q}$}
\newsavebox{\smlmatc}
\savebox{\smlmatc}{${\scriptstyle p(0)=\left(\protect\begin{smallmatrix}0.85\\ 1.5\protect\end{smallmatrix}\right) = p(2 \pi/3)}$}

\title{Symplectic integration of boundary value problems
}


\author{Robert I McLachlan         \and
        Christian Offen 
}


\institute{R. McLachlan \at
              Institute of Fundamental Sciences\\Massey University\\Palmerston North\\New Zealand\\
              \email{r.mclachlan@massey.ac.nz}           
           \and
           C. Offen \at
              Institute of Fundamental Sciences\\Massey University\\Palmerston North\\New Zealand\\
             	ORCiD: 0000-0002-5940-8057\\
			\email{c.offen@massey.ac.nz}
}




\maketitle

\begin{abstract}
Symplectic integrators can be excellent for Hamiltonian initial value problems. Reasons for this include
their preservation of invariant sets like tori, good energy behaviour, nonexistence of attractors,
and good behaviour of statistical properties. These all refer to {\em long-time} behaviour. They are
directly connected to the dynamical behaviour of symplectic maps $\varphi\colon M\to M$
on the phase space under iteration. Boundary value problems, in contrast, are posed for fixed (and often quite short)
times. Symplecticity manifests as a symplectic map $\varphi\colon M\to M'$ which is not iterated.
Is there any point, therefore, for a symplectic integrator to be used on a Hamiltonian boundary value problem? 
In this paper we announce results that symplectic integrators preserve bifurcations of Hamiltonian boundary value problems and that nonsymplectic integrators do not.
\keywords{Hamiltonian boundary value problems \and
bifurcations \and
periodic pitchfork \and
symplectic integration \and
geometric integration \and
singularity theory \and
catastrophe theory}
 \subclass{65L10  \and 65P10 \and 65P30 \and 37M15}	
\end{abstract}

\section{Motivation and introduction}

\subsection{The Bratu problem - an example of Hamiltonian boundary value problem}\label{subsec:Bratu}

As an instance of a Hamiltonian boundary value problem, let us consider the well-studied Bratu problem \cite{Mohsen201426}. In the one-dimensional case the Bratu problem refers to the steady-state solutions occurring in the following reaction-diffusion model of combustion
\[u_t = u_{xx}+Ce^u, \qquad u(t,0)=0=u(t,1).\]
Here $C > 0$ is a parameter. Steady-state solutions $x \mapsto u(x)$ fulfil the Dirichlet boundary value problem
\begin{equation}\label{eq:BratuODE}
u_{xx}+Ce^u=0, \quad u(0)=0=u(1).
\end{equation}
The left plot in figure \ref{fig:BratuSol} shows the two solutions for $C=1.5$. As $C$ increases the two solutions merge at a critical value $C^\ast \approx 3.513830719$ \cite[p.27]{Mohsen201426}. For $C>C^\ast$ no steady-state solutions exist. This can be seen in the bifurcation diagram displayed in the centre of figure \ref{fig:BratuBifur} where the $L^2$ norm of solutions is plotted against $C$.
\begin{figure}
\begin{center}
\includegraphics[width=0.32\textwidth]{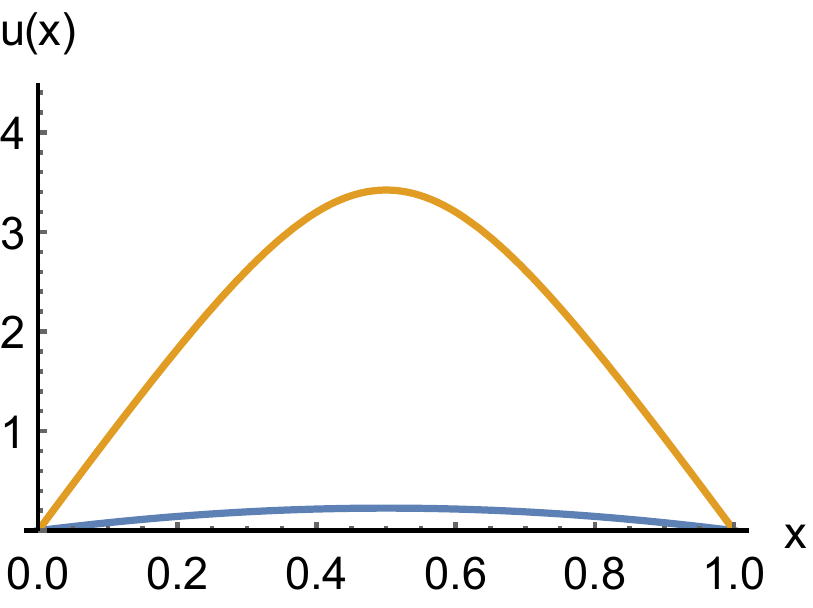}
\includegraphics[width=0.32\textwidth]{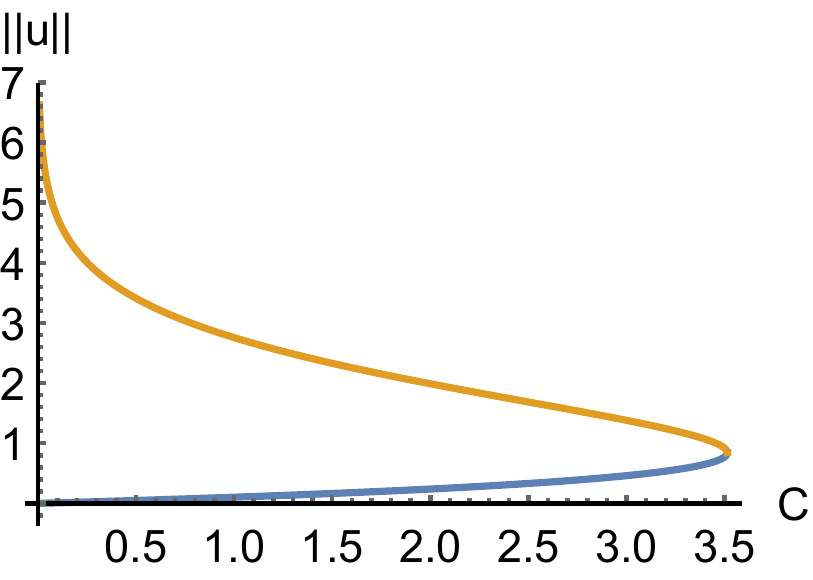}
\includegraphics[width=0.32\textwidth]{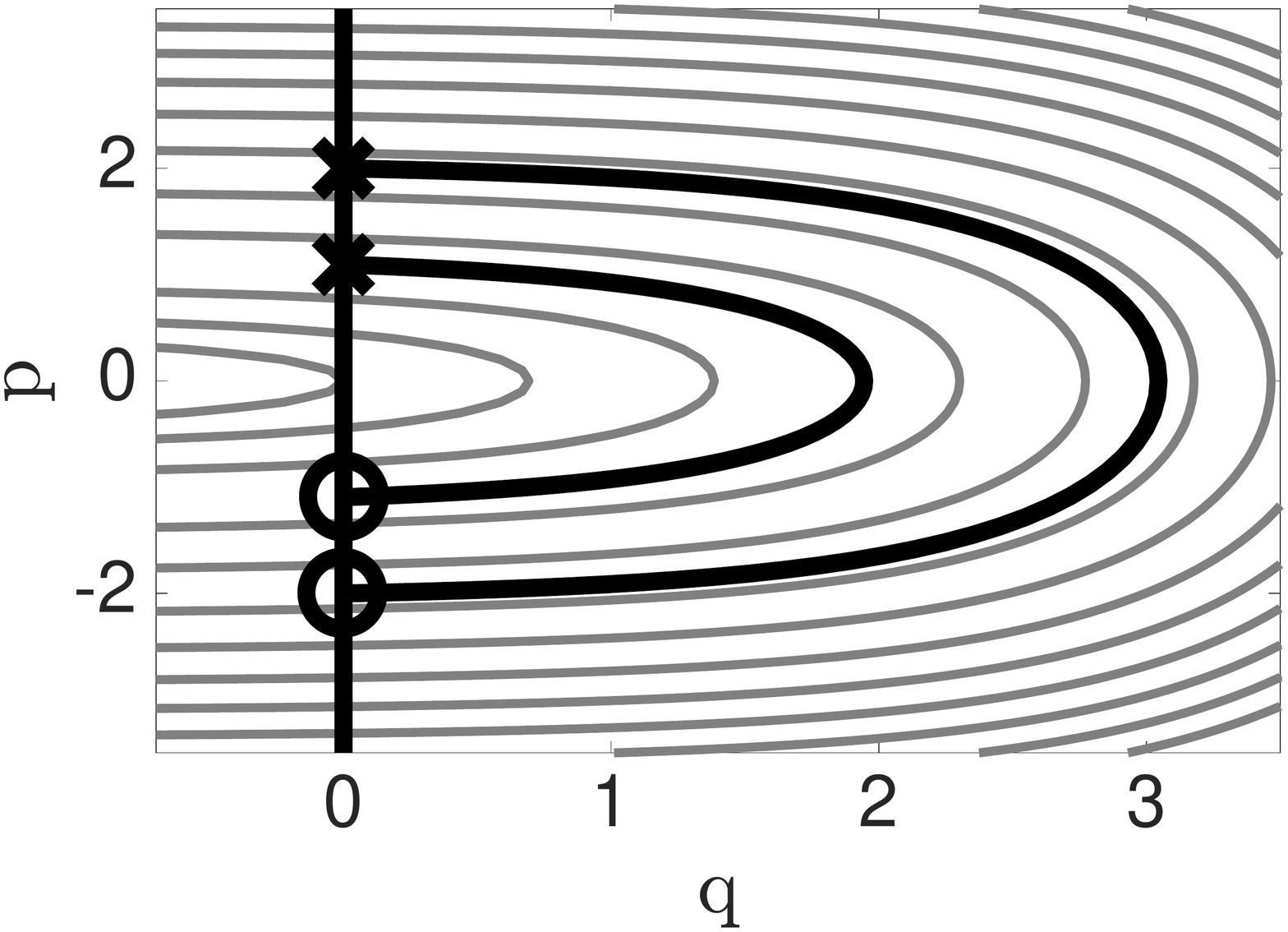}
\end{center}
\caption{The Bratu problem. Left plot: solutions to \eqref{eq:BratuODE} for $C=1.5$. Centre plot: bifurcation diagram to \eqref{eq:BratuODE} showing a fold bifurcation at $C\approx 3.51$. Right plot: Illustration of \eqref{eq:BratuODE} as a boundary value problem for the Hamiltonian system \eqref{eq:HamBratu}.}\label{fig:BratuSol}\label{fig:BratuBifur}\label{fig:BratuPhase}
\end{figure}
Let us view \eqref{eq:BratuODE} as a boundary value problem for a Hamiltonian system: consider the standard symplectic structure $\d q \wedge \d p$ on $T^\ast\R \cong \R^2$ and the Hamiltonian function
$H(q,p)=\frac 12 p^2 + C e^q$.
Hamilton's equations
\begin{align}\label{eq:HamBratu}
\dot q &= \; \; \, \nabla_p H(q,p) = p\\ \nonumber
\dot p &= -\nabla_q H(q,p) = - C e^q
\end{align}
provide a first-order formulation of the ODE \eqref{eq:BratuODE}. The boundary condition translates to $q(0)=0=q(1)$. The plot to the right in figure \ref{fig:BratuPhase} illustrates the boundary value problem in a phase portrait for the Hamiltonian system. The line $\Lambda=\{(0,p)\,|\, p \in \R\}$ corresponds to the homogeneous Dirichlet boundary condition. The boundary value problem is fulfilled if a motion starts on $\Lambda$ and returns to $\Lambda$ after time 1. For a value $C \in (0,C^\ast)$ two motions fulfilling the boundary value problem are illustrated as black curves starting at $\times$ and ending at $o$ in the plot.
%

\subsection{Purpose of the paper}\label{subsec:purpose}
To gain a good understanding of a parameter dependent boundary value problem a successful computation of the bifurcation diagram is necessary.
To draw valid conclusions from numerical results one has to make sure that the bifurcations in the boundary value problem for the exact flow are still present when the exact flow is perturbed by a numerical integrator and that no artificial bifurcations are introduced. It is, therefore, essential to
\begin{itemize}
\item
understand which kind of bifurcations can occur in a given problem class 
\item
and how to capture them numerically.
\end{itemize}

Moreover, bifurcations of high codimension act as \textit{organising centres} in the bifurcation diagram \cite[Part I, Ch.7]{Gilmore1993catastrophe}. This means a high codimensional bifurcation determines which bifurcations happen in a neighbourhood of the singular point.
It is, therefore, desirable to capture these correctly. Furthermore, bifurcation diagrams are often calculated using \textit{continuation methods}: a branch of bifurcations is followed numerically to find a bifurcation of higher codimension but these can only be detected correctly if they are not broken in the numerical boundary value problem. We conclude that preservation of the bifurcation behaviour is not only a goal in its own right but also crucial for computations.

The Bratu example is an instance of a {\em Lagrangian boundary value problem} for a Hamiltonian system. Indeed, it is a boundary value problem for the symplectic time-1-map of a Hamiltonian flow. Next to Dirichlet boundary conditions, Neumann-, Robin- and periodic boundary conditions are Lagrangian boundary conditions. We will refer to this class as {\em Hamiltonian boundary value problems}.
The authors attack the first task of the above bullet point list in \cite{bifurHampaper} linking bifurcations occurring in smooth parameterised families of Lagrangian boundary value problems for symplectic maps with catastrophe theory \cite{Arnold1,Gilmore1993catastrophe,lu1976singularity}. This applies to generic settings as well as to settings with ordinary or reversal symmetries. The conformal-symplectic symmetric case, which applies to homogeneous Hamiltonians and to the geodesic bifurcation problem in particular, is studied in \cite{obstructionPaper}.

Relevant to this paper are $A$-series bifurcations and $D$-series bifurcations. $A$-series bifurcations 
can be modelled as the qualitative change of the solution set to $\nabla g_\mu(x) =0$ with $g_\mu(x) = x^{n+1} + \sum_{j=1}^{n-1} \mu_j x^j$ as the parameter $\mu$ is varied. They are denoted by $A_n$ and the first instances are called {\em fold} ($n=2$), {\em cusp} ($n=3$), {\em swallowtail} ($n=4$). The first two $D$-series bifurcations are given by $ g_\mu(x,y)=x^3\pm xy^2+\mu_3 (x^2+y^2)+\mu_2 y + \mu_1 x$. They are denoted by $D_4^{\pm}$ and are called {\em hyperbolic umbilic bifurcation} ($D_4^{+}$) and {\em elliptic umbilic bifurcation} $D_4^{-}$. The bifurcations $A_2$, $A_3$, $A_4$, $D_4^+$, $D_4^-$ are (in an appropriate equivalence relation) the only bifurcations which occur in generic Hamiltonian boundary value problems with up to three parameters \cite{bifurHampaper}.

The purpose of this paper is to announce results the authors obtained for the second objective of the bullet point list, {i.e.\ the preservation of bifurcations under discretisation.}

Symplecticity in Hamiltonian boundary value problems does not seem to have been addressed in the literature, even in very detailed numerical studies like \cite{Beyn2007,GalanVioque2014}.
The AUTO software \cite{AutoManual} is based on Gauss collocation, which is symplectic when the equations are presented in canonical variables. The two-point boundary-value codes MIRKDC \cite{MIRKDC}, TWPBVP \cite{TWPBVP} and TWPBVPL \cite{BashirAli199859,Cash200681} are based on non-symplectic Runge-Kutta methods. MATLAB’s {\tt bvp4c} uses 3-stage Lobatto IIIA \cite{KierzenkaBvP4c}, which is not symplectic. Note that symplectic integration sometimes requires the use of implicit methods. For initial value problems, these are typically computationally more expensive than explicit methods. However, for boundary value problems solved in the context of parameter continuation, this distinction largely disappears as excellent initial approximations are available.
Other approaches like \cite{grass2015optimal,UECKER2016} use the code TOM \cite{Mazzia2006,MAZZIA2009723,Mazzia2004}. Moreover, Hamiltonian boundary value methods, designed to preserve Hamiltonians up to any fixed polynomial order, can be used where energy conservation is essential \cite{Amodio2015,BRUGNANO2015650}.


\section{Broken gradient-zero bifurcations}\label{sec:brokenBifur}


\begin{definition}[symplectic integrator]
A symplectic integrator assigns to a time-step-size $h>0$ (discretisation parameter) and a Hamiltonian system a symplectic map which approximates the time-$h$-map of the Hamiltonian flow of the system (which is symplectic).
\end{definition}

\begin{remark}
For a finite sequence of positive time-step-sizes $h_1,\ldots,h_N$ summing to $\tau$ the composition of all time-$h_j$-map approximations obtained by a symplectic integrator yields an approximation to the Hamiltonian-time-$\tau$-map, which is a symplectic map.
\end{remark}

The solutions to a family of Hamiltonian boundary value problems on $2n$-dimensional manifolds locally corresponds to the roots of a family of $\R^{2n}$-valued function defined on an open subset of $\R^{2n}$. For a Hamiltonian boundary value problems these maps are exact, i.e.\ each arises as the gradient of a scalar valued map \cite{bifurHampaper}. Consider a family of Hamiltonian boundary value problems
and consider an approximation of the Hamiltonian-time-$\tau$-map by an integrator. Roughly speaking, two map-families are (right-left-) equivalent if they coincide up to reparametrisation and parameter dependent changes of variables in the domain and target space.
If the family of maps corresponding to the approximated problems is equivalent to the family of maps for the exact problem then we say \textit{the integrator preserves the bifurcation diagram of the problem}.


\begin{prop}\label{prop:capturebifur}
A symplectic integrator with any fixed (but not necessarily uniform) step-size, applied to any autonomous or nonautonomous Hamiltonian boundary value problem, preserves bifurcation diagrams of generic bifurcations of any codimension for sufficiently small maximal step-sizes.
\end{prop}


\begin{proof}
All generically occurring singularities in Lagrangian boundary value problems for symplectic maps are non-removable under small symplectic perturbations of the map. The statement follows because a Hamiltonian diffeomorphism which is slightly perturbed by a symplectic integrator is a symplectic map near the exact flow map. 
\end{proof}


Proposition \ref{prop:capturebifur} implies that using a symplectic integrator to solve Hamilton's equations in order to solve a Lagrangian boundary value problem we obtain a bifurcation diagram which is qualitatively correct even when computing with low accuracy and not preserving energy.
In contrast, nonsymplectic integrators do not preserve all bifurcations, even for arbitrary small step-sizes. However, they do preserve the simplest class of $A$-series bifurcations, i.e.\ folds, cusps, swallowtails, butterflies,....

\begin{prop}\label{prop:Aseriespersist} 
A non-symplectic integrator with any fixed (but not necessarily uniform) step-size, applied to any autonomous or nonautonomous Hamiltonian Lagrangian boundary value problem, preserves bifurcation diagrams of generic $A$-series singularities for sufficiently small maximal step-sizes. However, each non-symplectic integrator breaks the bifurcation diagram of all generic $D$-series singularities for any positive maximal step-size.
\end{prop}


\begin{proof}
Passing to a generating function of the Lagrangian boundary value problem, as explained in  \cite{bifurHampaper}, solutions locally correspond to the roots of an $\R^{2n}$-valued function $F$ defined on an open subset of $\R^{2n}$ where $F$ arises as the gradient of a scalar valued map.
A discretisation of the flow map corresponds to a smooth perturbation $\tilde F$ of $F$. The perturbed map $\tilde F$ arises as the gradient of a scalar valued map if and only if the discretisation of the flow is symplectic.
$A$-series bifurcations are stable in the roots-of-a-function and, therefore, persists under any small, smooth perturbation.
$D$-series singularities, however, decompose into $A$-series singularities under arbitrarily small smooth perturbations, which do not respect the gradient structure of the problem \cite{numericalPaper}.
\end{proof}

\begin{remark}
For the fold bifurcation in the Bratu problem (figure \ref{fig:BratuBifur}) the proposition says that any integrator with fixed step-size will capture the bifurcation correctly, i.e.\ the obtained bifurcation diagram will qualitatively look the same as the plot in the centre of figure \ref{fig:BratuBifur}.
\end{remark}

\subsection{Breaking of hyperbolic and elliptic umbilic bifurcations}\label{subsec:breakbifur}


Let us take a closer look at the first two $D$-series bifurcations and obtain models for how $D$-series bifurcations in Hamiltonian boundary value problems break when using a non-symplectic integrator.

A universal unfolding of the hyperbolic umbilic singularity $D_4^+$ with parameter $\mu$ is given by
\[g_\mu(x,y)=x^3+xy^2+\mu_3 (x^2+y^2)+\mu_2 y + \mu_1 x.\]
The plot to the left in figure \ref{fig:breakD4p} shows the level bifurcation set to the problem $\nabla g_\mu(x,y)=0$. It consists of those points $\mu$ in the parameter space for which a bifurcation occurs in the phase space, i.e.\ there exists a point $(x,y)$ in the phase space such that $\nabla g_\mu(x,y)=0$ and $\det \mathrm{Hess}\, g_\mu(x,y) =0$.
The plot to the right shows the level bifurcation set to the perturbed problem $\nabla g_\mu (x,y) +f_{\epsilon}(x,y) =0$
%
for $\e \not=0$ near $0$ and a smooth family of maps $f_{\epsilon}\colon \R^2\to\R^2$ with $f_0=0$ such that $f_{\epsilon} \not= \nabla h_\e$ for any $h_\e \colon \R^2\to \R$ unless $\e =0$. Here $\D (\nabla g_\mu + f_{\epsilon})(x,y)$ denotes the Jacobian matrix of the map $(x,y) \mapsto  (\nabla g_\mu + f_{\epsilon})(x,y)$.

Each point in the sheets corresponds to a fold singularity ($A_2$) and points on edges to cusp singularities ($A_3$). At parameter values where the sheets self-intersect there are two simultaneous fold singularities in the phase space. In the unperturbed system two lines of simultaneous folds merge with a line of cusps to a hyperbolic umbilic point \cite[I.5]{Gilmore1993catastrophe}. In the perturbed picture the line of cusps breaks into three segments and two swallowtail points ($A_4$) occur where two lines of cusps merge with a line of simultaneous folds. Notice that there are no swallowtail points in the unperturbed level bifurcation set.

\begin{figure}
\begin{center}
\includegraphics[width=0.49\textwidth]{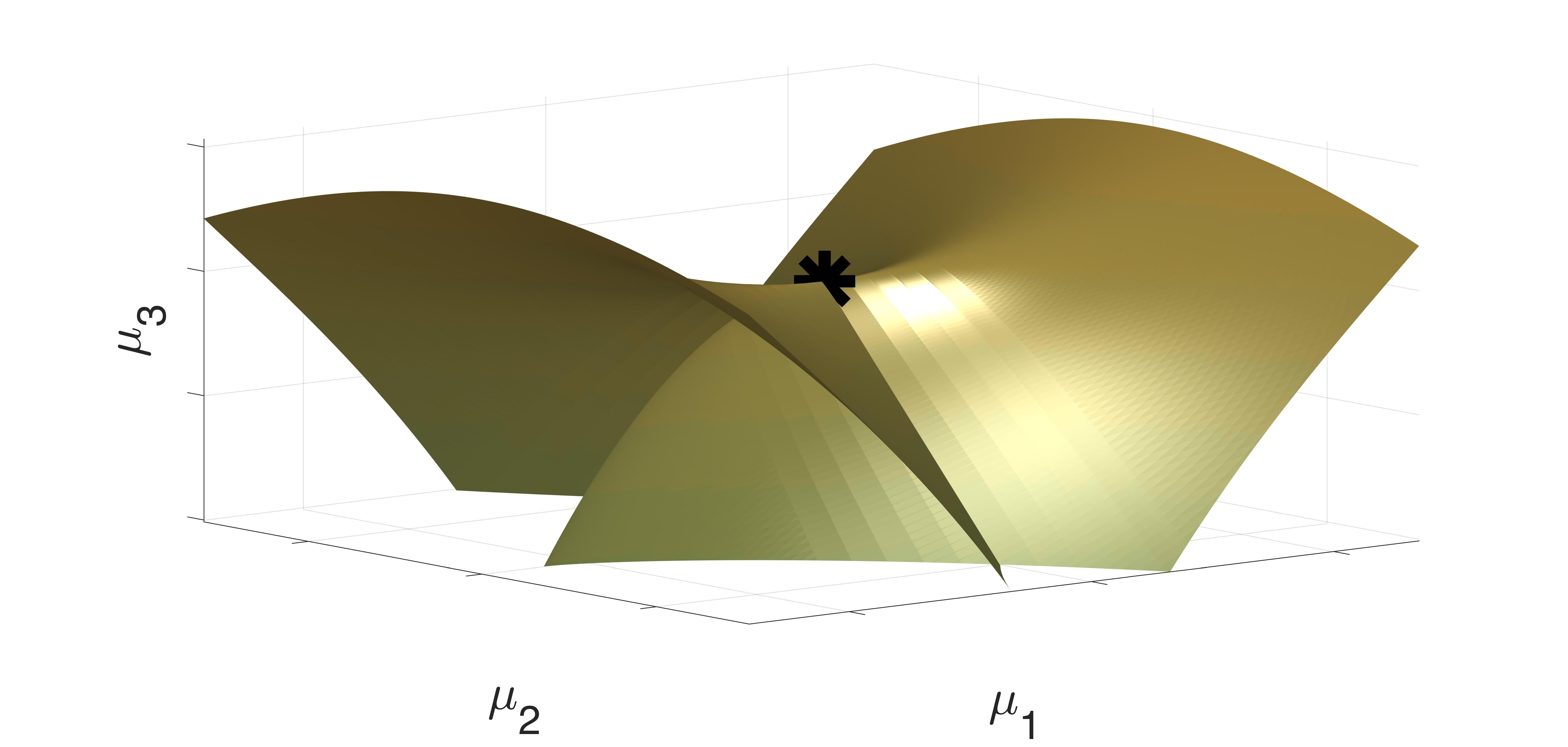}
\includegraphics[width=0.49\textwidth]{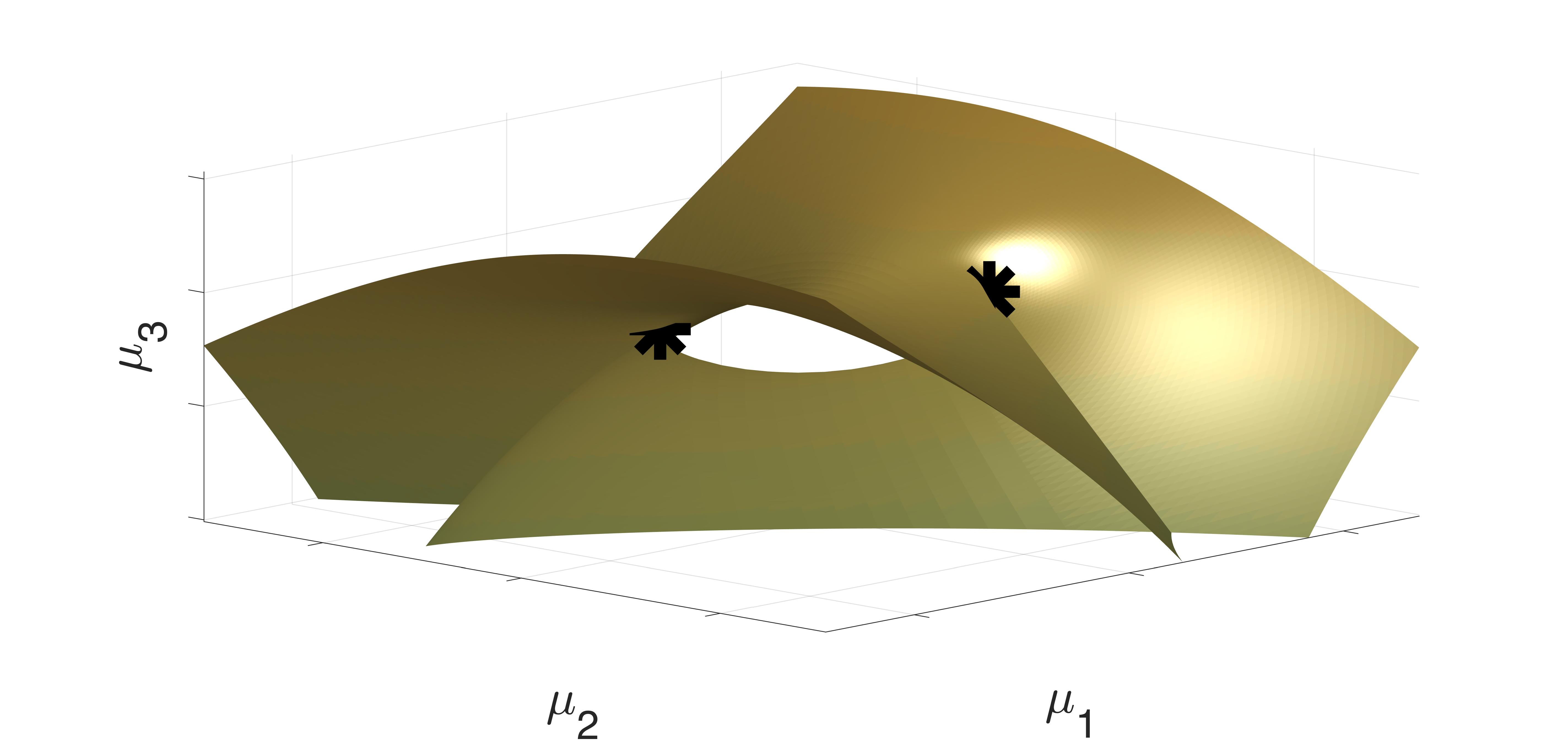}
\end{center}
\caption{The plots show those configurations of the parameters $\mu_1,\mu_2,\mu_3$ for which the problem $\nabla g_\mu(x,y)=0$ or $(\nabla g_\mu + f_\e)(x,y)=0$ becomes singular. Imagine moving around the parameter $\mu$ and watching the solutions bifurcating in the phase space. 
As $\mu$ crosses a sheet two solutions merge and vanish or are born (fold - $A_2$).
For $\mu$ in the intersection of two sheets there are two simultaneous fold singularities at different positions in the phase space. Crossing an edge three solutions merge into one (or vice versa).
Points contained in an edge correspond to cusp singularities.
At the marked point in the left plot of the unperturbed problem there is a hyperbolic umbilic singularity. Moving the parameter $\mu$ upwards along the $\mu_3$ axis through the singular point four solutions merge and vanish. In the perturbed version to the right the hyperbolic umbilic point decomposes into two swallowtail points. While the left plot {illustrates the behaviour of} a symplectic integrator {which will} correctly show a hyperbolic umbilic bifurcation $D_4^+$, the right plot {illustrates the behaviour of} a non-symplectic integrator {which will} incorrectly show two nearby swallowtail bifurcations ($A_4$).
}\label{fig:breakD4p}
\end{figure}



Figure \ref{fig:breakD4m} shows a level bifurcation set of an elliptic umbilic singularity $(D_4^-)$ and a generically perturbed version of the gradient-zero problem with a map that does not admit a primitive. Here we use the universal unfolding
\[g_\mu(x,y) =x^3-xy^2 + \mu_3 (x^2+y^2) +\mu_2 y + \mu_1 x.\]
We see that in the perturbed picture the lines of cusps fail to merge such that there is no elliptic umbilic point but only folds and cusp bifurcations.

\begin{figure}
\begin{center}
\includegraphics[width=0.49\textwidth]{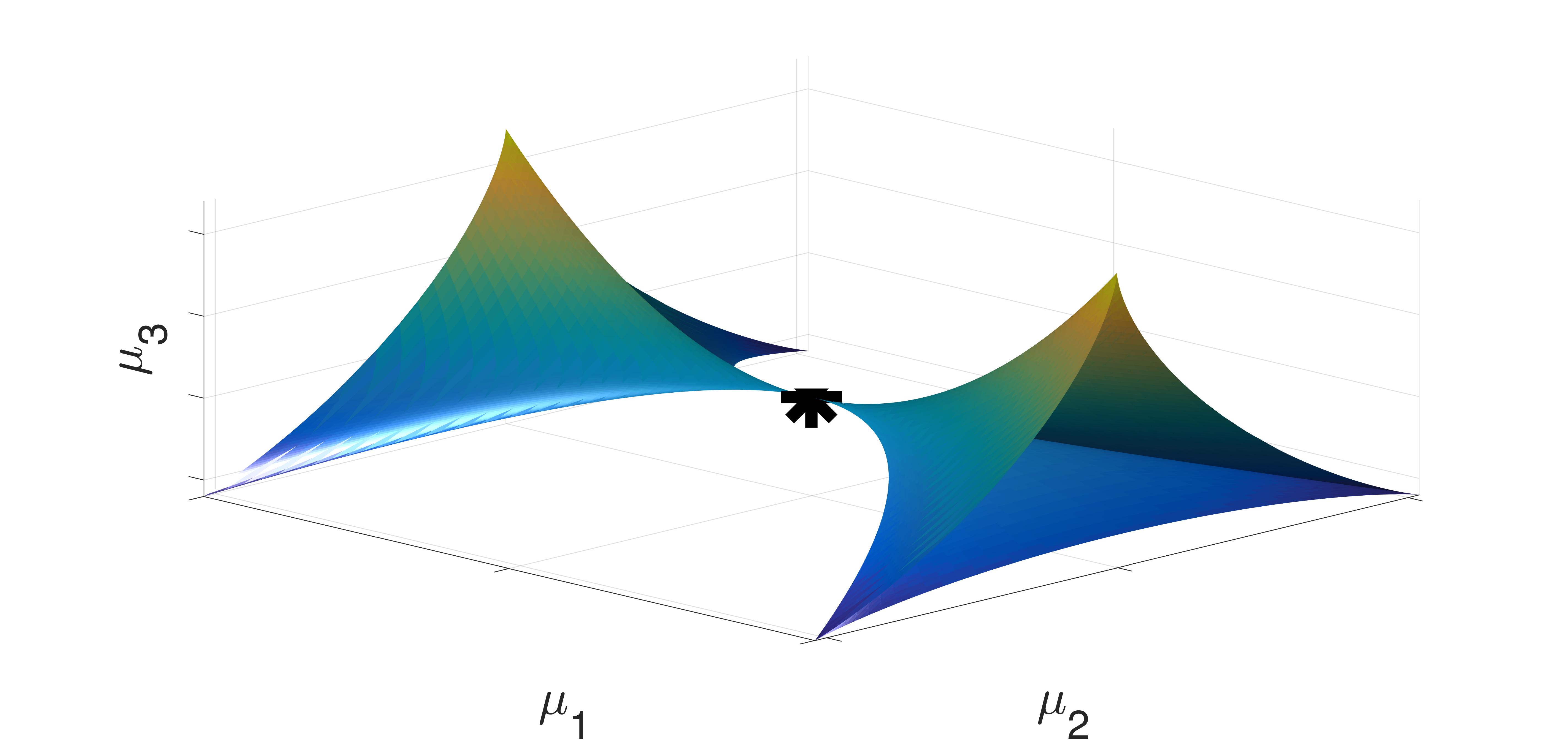}
\includegraphics[width=0.49\textwidth]{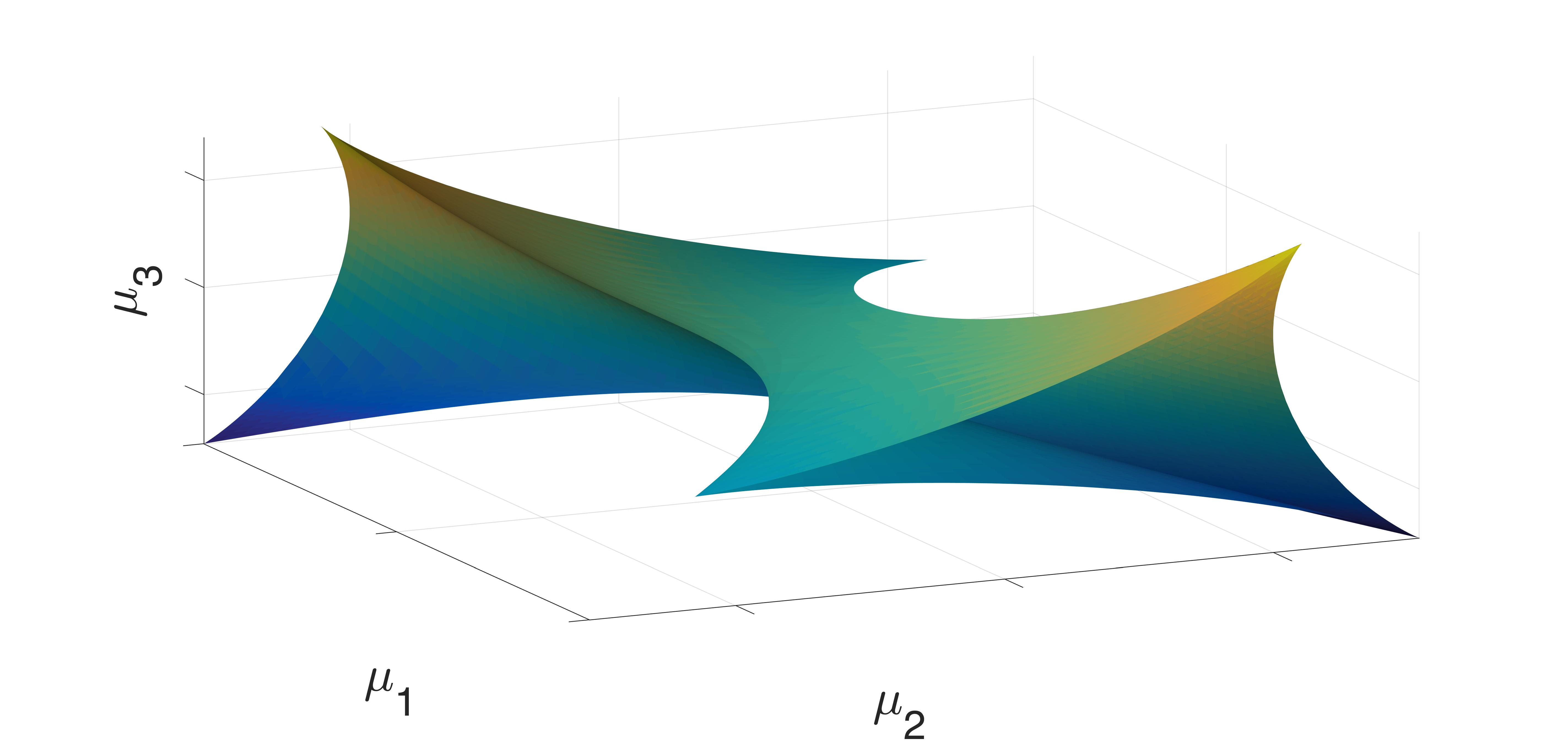}
\end{center}
\caption{
An exact (left) and a perturbed (right) version of the bifurcation level set of an elliptic umbilic singularity $D_4^-$. The elliptic umbilic point marked by an asterisk in the left plot is not present in the right plot where three lines of cusps fail to merge. The left figure {illustrates} a bifurcation diagram obtained by a symplectic integrator correctly showing a $D_4^-$ singularity, the plot to the right {illustrates} using a non-symplectic integrator incorrectly showing no elliptic umbilic point.}\label{fig:breakD4m}
\end{figure}


Indeed, the authors prove in \cite{numericalPaper} that the behaviour shown in figure \ref{fig:breakD4p} and \ref{fig:breakD4m} is universal. This means
%
in any Hamiltonian boundary value problem with a generic hyperbolic or elliptic bifurcation any symplectic integrator will show a bifurcation diagram as on the left of figures \ref{fig:breakD4p} and \ref{fig:breakD4m} while any integrator which breaks the symplectic structure of the problem will show incorrect bifurcation diagrams which qualitatively look like those on the right of figures \ref{fig:breakD4p} and \ref{fig:breakD4m}.

\subsection{Example. H{\'e}non-Heiles Hamiltonian system}
Consider the H{\'e}non-Heiles Hamiltonian on the phase space $T^\ast\R^2 \cong \R^2\times\R^2$ for the parameter value $-10$, i.e.\
\begin{equation}\label{eq:HenonHeilesHam}
H(q,p)=\frac 12 \|p\|^2+\frac 12 \|q\|^2-10\left(q_1^2 q_2-\frac{q_2^3}3 \right).
\end{equation}
In \eqref{eq:HenonHeilesHam} the norm $\| . \|$ denotes the euclidean norm on $\R^2$. We obtain a symplectic map $\phi$ by integrating Hamilton's equations
\[
\dot q = \nabla_p H(q,p),\quad \dot p = -\nabla_q H(q,p)
\]
up to time $\tau=1$ using the 2nd order symplectic St\"ormer-Verlet scheme with 10 time-steps.
Consider the following Dirichlet-type problem for the given Hamiltonian system: a point $(q,p)$ is a solution to the boundary value problem if and only if
\[
q=q^\ast \quad \text{and} \quad \phi^Q(q,p)=Q^\ast,
\]
where $\phi^Q$ denotes the $Q$-component of $\phi$. Let the boundary values $q^\ast$ and $Q^\ast$ be the parameters of the system (in contrast to the example presented in section \ref{subsec:Bratu} where the parameter was in the Hamiltonian). 
To reduce dimensionality we fix the first component of the start value, i.e.\ we set $(q^\ast)^1=0$. The level bifurcation set, i.e.\ the set of points in the parameter space at which a bifurcation occurs in a chosen subset $U$ of the phase space, is given as
\[ \{ (q^2,\phi^Q(0,q^2,p_1,p_2)) \; | \; \det D_p \phi^Q(0,q^2,p_1,p_2)=0,\, (p_1,p_2)\in U \}. \]

Figure \ref{fig:D4m} shows the level bifurcation set of the problem near an elliptic umbilic singularity $D_4^-$. Derivatives of the symplectic approximation to $\phi$ were obtained using automatic differentiation. The $D$-series bifurcation was found numerically by solving $(q^2,p_1,p_2) \mapsto D_p \phi^Q(0,q^2,p_1,p_2)=0$. It is correctly resolved. 
\begin{figure}
\begin{center}
\includegraphics[width=0.5\textwidth]{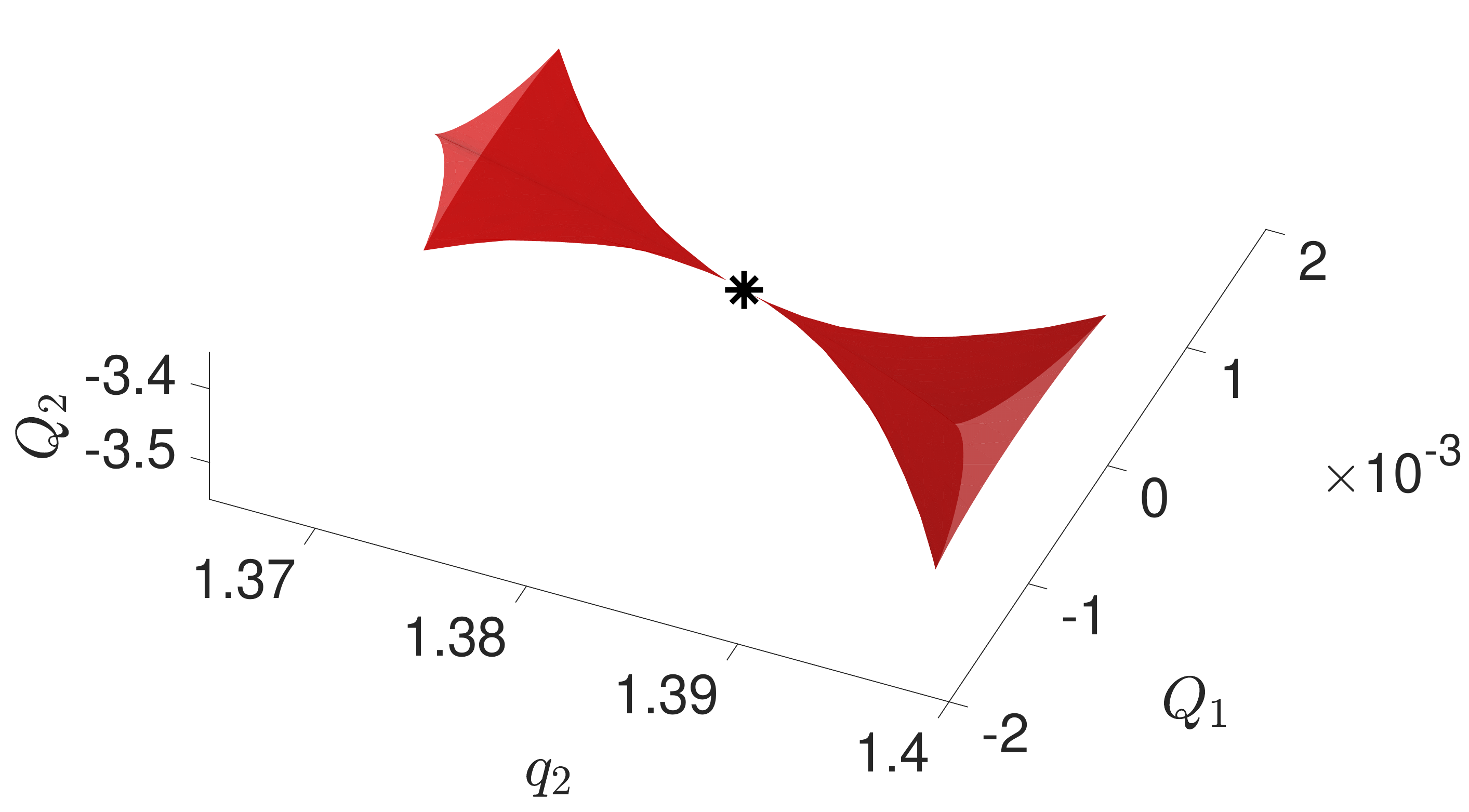}
\end{center}
\caption{Elliptic umbilic $D_4^-$ in a family of Dirichlet problems for the numerical time-1-map of the H{\'e}non-Heiles system \eqref{eq:HenonHeilesHam} where the boundary values are parameters and $(q^\ast)^1=0$ is fixed to reduce dimensionality. The numerical flow was obtained using the 2nd order symplectic St\"ormer-Verlet scheme with 10 time-steps. The asterisk denotes the calculated position of the elliptic umbilic singularity. Derivatives were obtained using automatic differentiation.}\label{fig:D4m}
\end{figure}
By the considerations of section \ref{subsec:breakbifur}, the bifurcation would break if we used a non-symplectic discretisation for the flow map of $H$ instead of the symplectic St\"ormer-Verlet scheme. See \cite[Appendix B]{numericalPaper} for a numerical experiment.

\section{Capturing periodic pitchfork bifurcations in integrable systems}\label{sec:capturePitchfork}

The minimal amount of parameters in a family of problems such that a singularity is generic, i.e.\ unremovable under small perturbations, depends on the class of systems considered. For example, we have shown that a $D^{\pm}_4$ singularity occurs generically in Hamiltonian boundary value problems with 3 parameters.
In a boundary value problem for a flow map without any extra (e.g.\ symplectic) structure a $D^{\pm}_4$ singularity needs 4 parameters to become generic.
Restricting the class of systems further, e.g.\ to those with certain symmetries and/or integrals of motion, the count of required parameters can change. Here we consider a special singularity which occurs generically in 1-parameter families of symmetrically separated Lagrangian boundary value problems for completely integrable Hamiltonian systems, e.g.\ planar, autonomous systems.

\subsection{Introduction and the effects of discretisation}
Consider a Lagrangian submanifold $\Lambda$ in the phase space of a Hamiltonian system. The manifold $\Lambda$ defines a symmetrically separated Lagrangian boundary value problem: a motion is a solution if and only if it starts and ends after a fixed time on $\Lambda$. Homogeneous Dirichlet boundary conditions as in figure \ref{fig:BratuPhase} are instances of such a boundary condition. As the authors prove in \cite[Thm. 3.2]{bifurHampaper}, a periodic pitchfork bifurcation (see the second plot in figure \ref{fig:PPcyclic1415newcoords}) is a generic phenomenon in 1-parameter families of boundary value problems in \textit{completely integrable}\footnote{If the phase space dimension is $2n$ then there are $n$ functionally independent integrals of motion.} Hamiltonian systems with symmetrically separated Lagrangian boundary conditions. 

In \cite{bifurHampaper} the authors reveal how the completely integrable structure and the structure of the boundary conditions induce a $\Z/2\Z$-symmetry in the generating function of the problem family. The singular point of a pitchfork bifurcation is unfolded under the presence of a $\Z/2\Z$-symmetry to a pitchfork bifurcation. The corresponding critical-points-of-a-function problem is defined by the family $(x^4+\mu_2 x^2)_{\mu_2}$. Unfolding without the $\Z/2\Z$-symmetry leads, however, to the normal form of a cusp bifurcation which is defined by the family  $(x^4+\mu_2 x^2+\mu_1x)_{\mu_1,\mu_2}$. The effect of the symmetry breaking parameter $\mu_1$ is illustrated in figure \ref{fig:pitchforkcusp}: we see how the pitchfork bifurcation, which is present for $\mu_1=0$, breaks if $\mu_1 \not=0$.

Consider a completely integrable Hamiltonian boundary value problem. Approximating Hamilton's equations introduces the discretisation parameter $h$ as an additional parameter. The discretisation does not respect the completely integrable structure. 
\begin{figure}
\begin{center}
\includegraphics[width=0.4\textwidth]{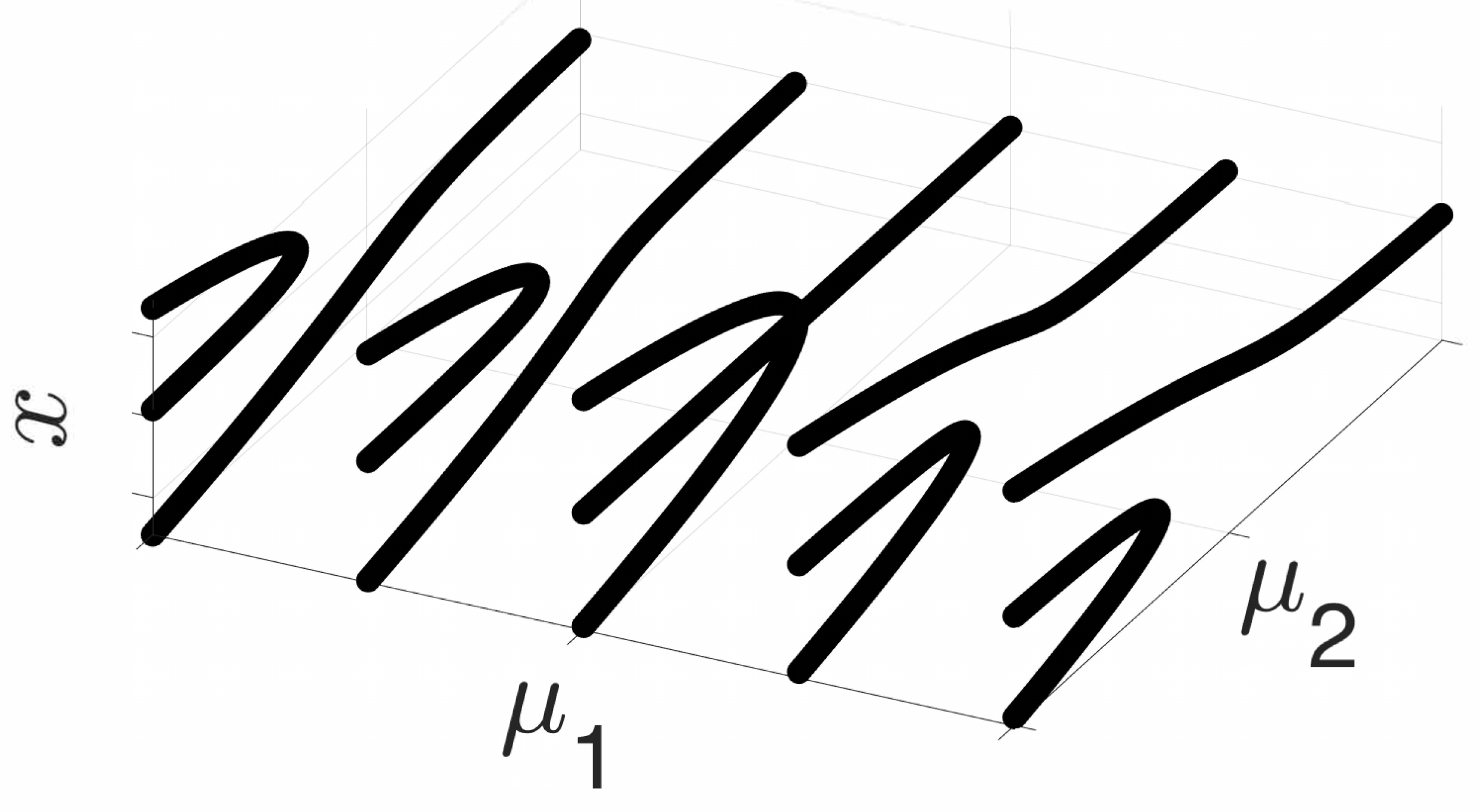}
\end{center}
\caption{The figure shows the critical point set of the model cusp $x^4+\mu_2 x^2+ \mu_1 x$ over the $\mu_1$/$\mu_2$-parameter space for selected values of $\mu_1$. For $\mu_1=0$ we see a pitchfork bifurcation.}\label{fig:pitchforkcusp}
\end{figure}
If the order of accuracy of the integrator is $k$ then, generically, the power of the step-size $h^k$ acts like the unfolding parameter $\mu_1$ in figure \ref{fig:pitchforkcusp}. We say the pitchfork is \textit{broken up to the order of accuracy of the integrator}.
This means in a generic setting symplecticity of an integrator cannot be expected to improve the numerical capturing of periodic pitchfork bifurcations because the periodic pitchfork bifurcation is related to the integrable structure rather than to symplecticity. 
However, in many important cases, symplecticity does help because symplectic integrators preserve a modified Hamiltonian exponentially well \cite[IX]{GeomIntegration} and are, therefore, guaranteed to capture at least this part of the integrable structure very well. In the planar case this means the whole integrable structure is captured exponentially well by symplectic integrators. Here, the discretisation parameter does not enter generically but unfolds the pitchfork bifurcation to a family of nearly perfect pitchforks. These pitchforks are broken only up to exponential order in $-h^{-1}$. The same is true in higher dimensional examples if additional integrals/symmetries are captured because they are, e.g.\ affine linear.
%
%
%
To which extent the completely integrable structure of a system is present in the numerical flow determines how well a pitchfork bifurcation is captured. 


\subsection{Numerical examples}

We analyse the generic 1-parameter family of planar Hamiltonian boundary value problems
\begin{equation}\label{eq:numEx}
H_\mu(q,p) = p^2 + 0.1p^3 - 0.01 \cos(p)  + q^3-0.01q^2+\mu q, \quad q(0)=0.2 = q(1.7). \end{equation}

Figure \ref{fig:pitchfork2d} shows how a pitchfork bifurcation of \ref{eq:numEx} is captured by the symplectic St\"ormer-Verlet  method with 14, 21 and 28 steps. The breaking in the bifurcation for 28 steps is visible in a close-up of the bifurcation diagram. Notice the different scaling of the axes in the plots. We see that only few time-steps are needed to capture the bifurcation very well. The strong improvement of the shape of the pitchfork as the the amount of steps is increased indicates a convergence to the correct shape which is better than polynomial.

\begin{figure}
\begin{center}
\includegraphics[width=0.32\textwidth]{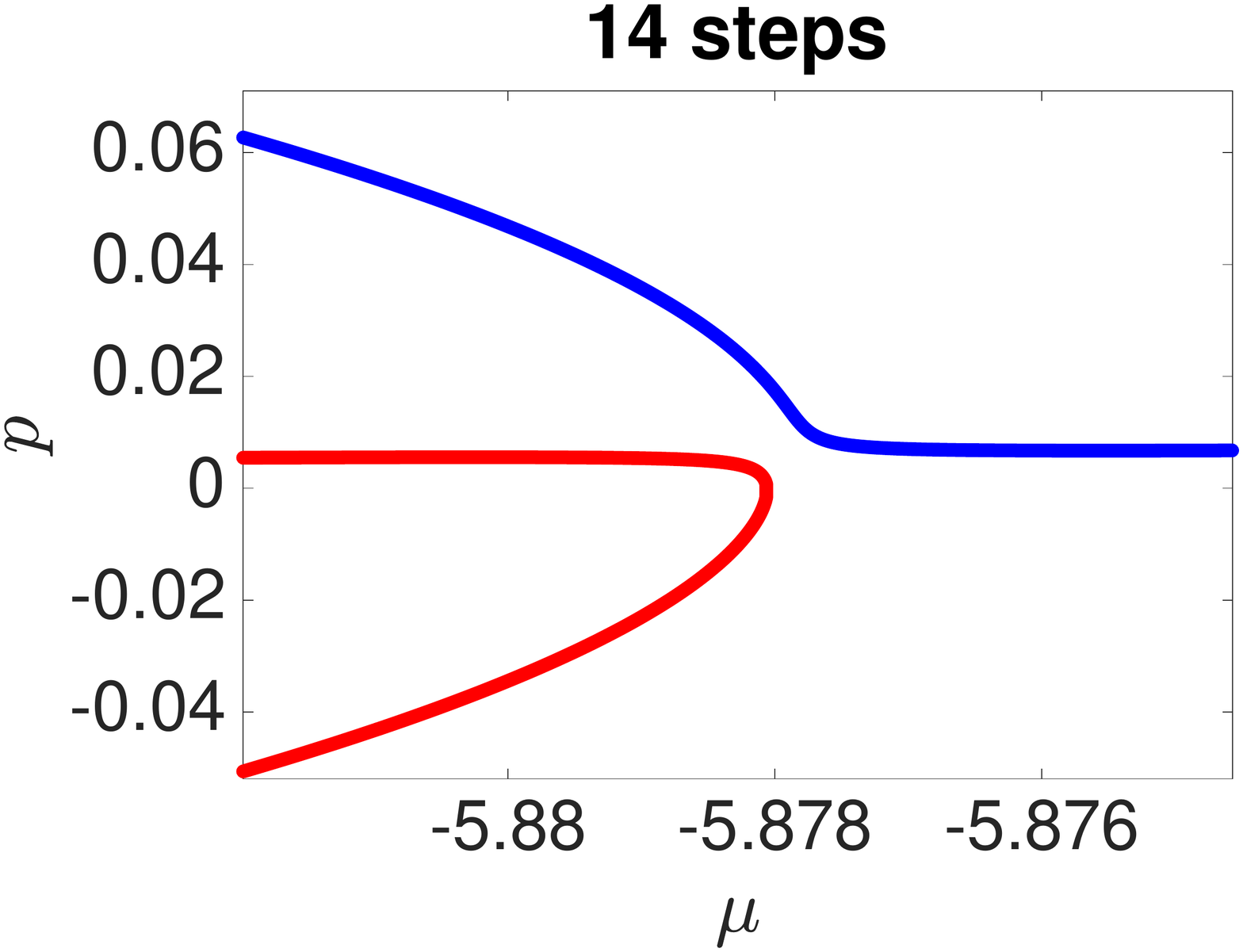}
\includegraphics[width=0.32\textwidth]{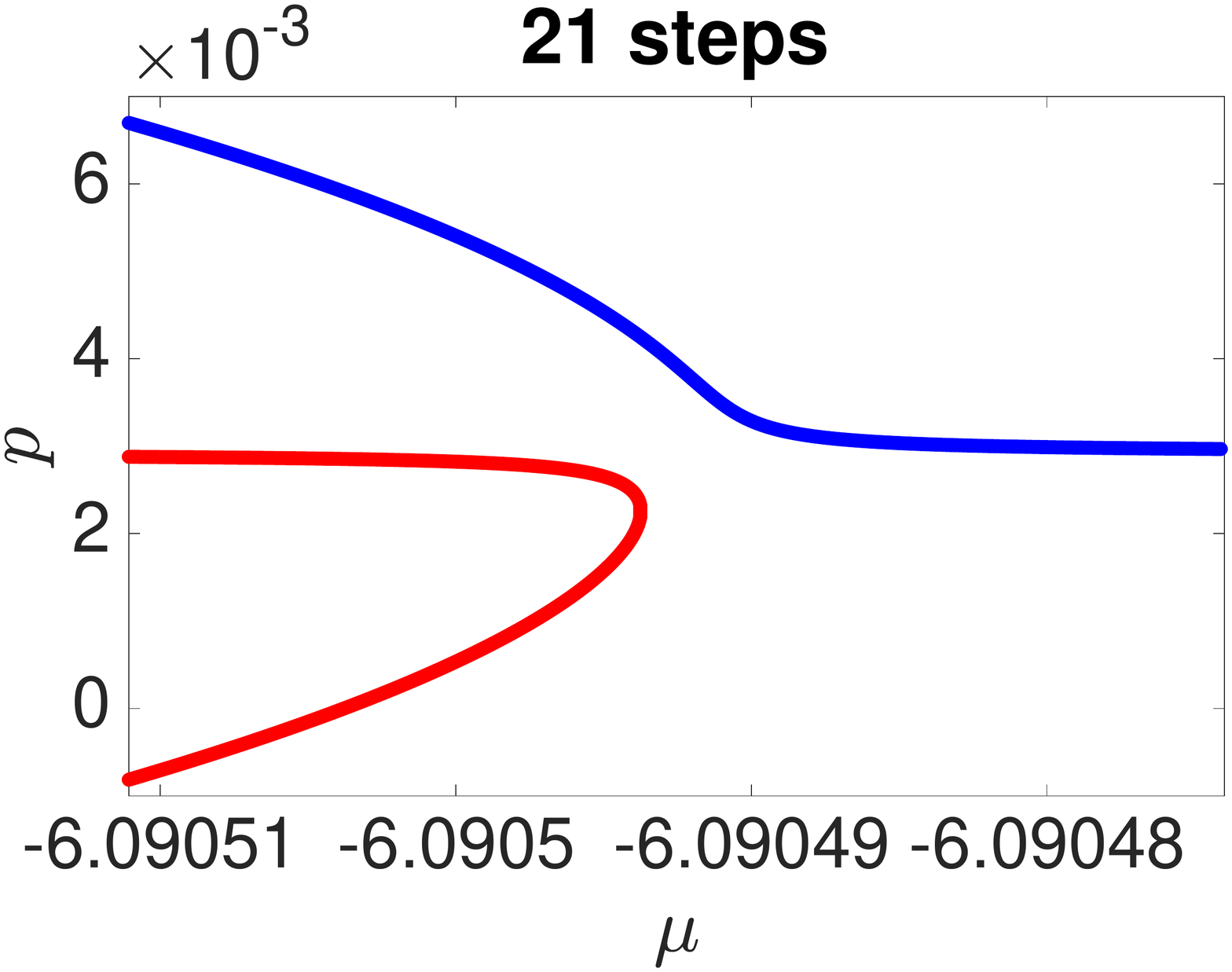}
\includegraphics[width=0.32\textwidth]{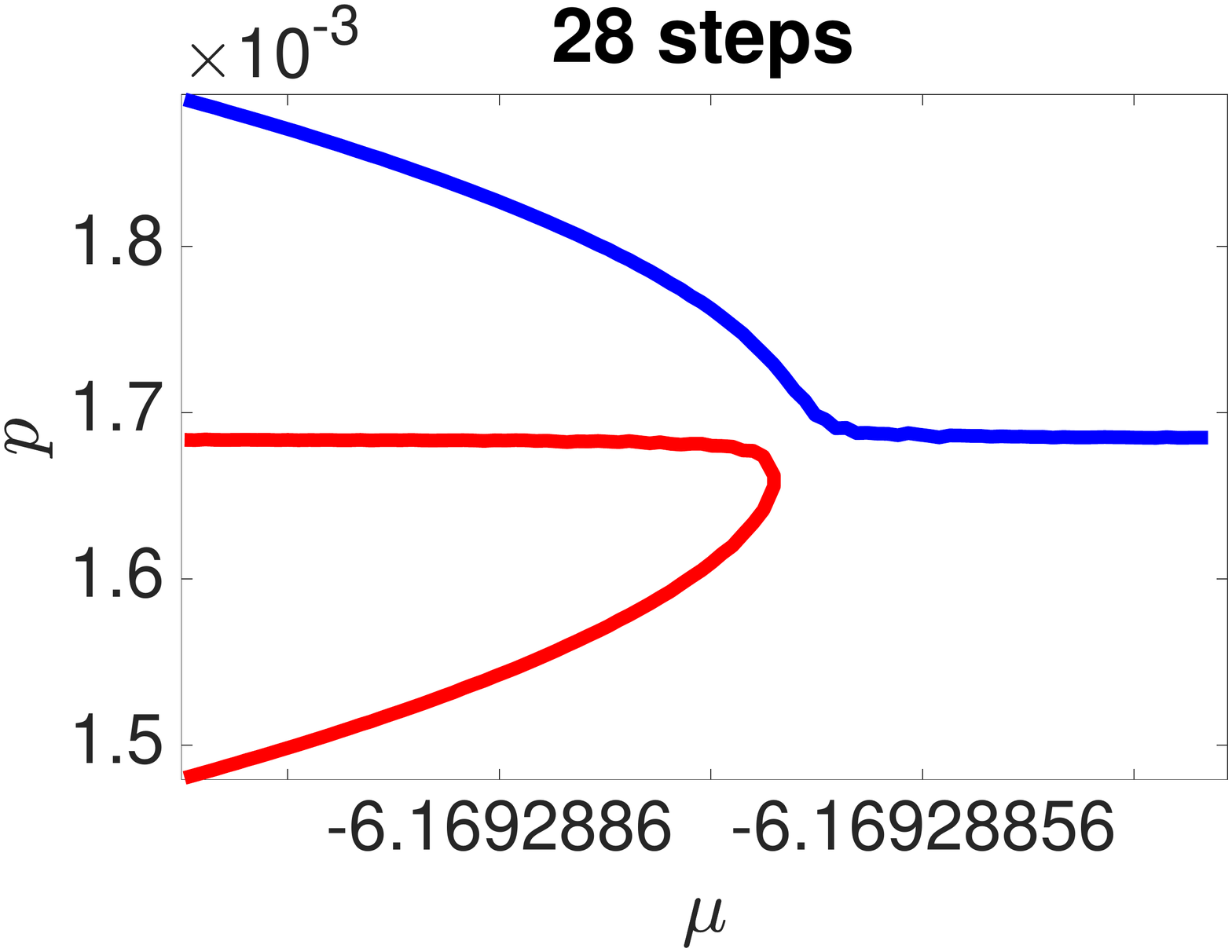}
\end{center}
\caption{Bifurcation diagrams for \eqref{eq:numEx} solved with the symplectic St\"ormer-Verlet method using 14, 21 and 28 steps. Notice different scaling of the axes. {The shape of the pitchfork bifurcation is captured exponentially well.}}\label{fig:pitchfork2d}
\end{figure}

Let us compare the observations from figure \ref{fig:pitchfork2d} with a non-symplectic integrator of the same order of accuracy. To understand what is happening in the latter case we compute a larger part of the bifurcation diagram using the non-symplectic, 2nd order Runge-Kutta method (RK2). The upper and middle branch of the pitchfork bifurcation do not exist in the numerical bifurcation diagram until we calculate with more than 25 steps (figure \ref{fig:pitchforkRK}). With 100 steps the bifurcation is recognisable and with 400 steps its break can only be seen in a close-up plot and the quality of the capture is comparable with the 14-steps St\"ormer-Verlet integration from figure \ref{fig:pitchfork2d}.
As the computational costs per step for both methods do not differ significantly on separable Hamiltonian systems we conclude that the symplectic St\"ormer-Verlet method (figure \ref{fig:pitchfork2d}) performs remarkably better then the non-symplectic method RK2 (figure \ref{fig:pitchforkRK}).

\begin{figure}
\begin{center}
\includegraphics[width=0.325\textwidth]{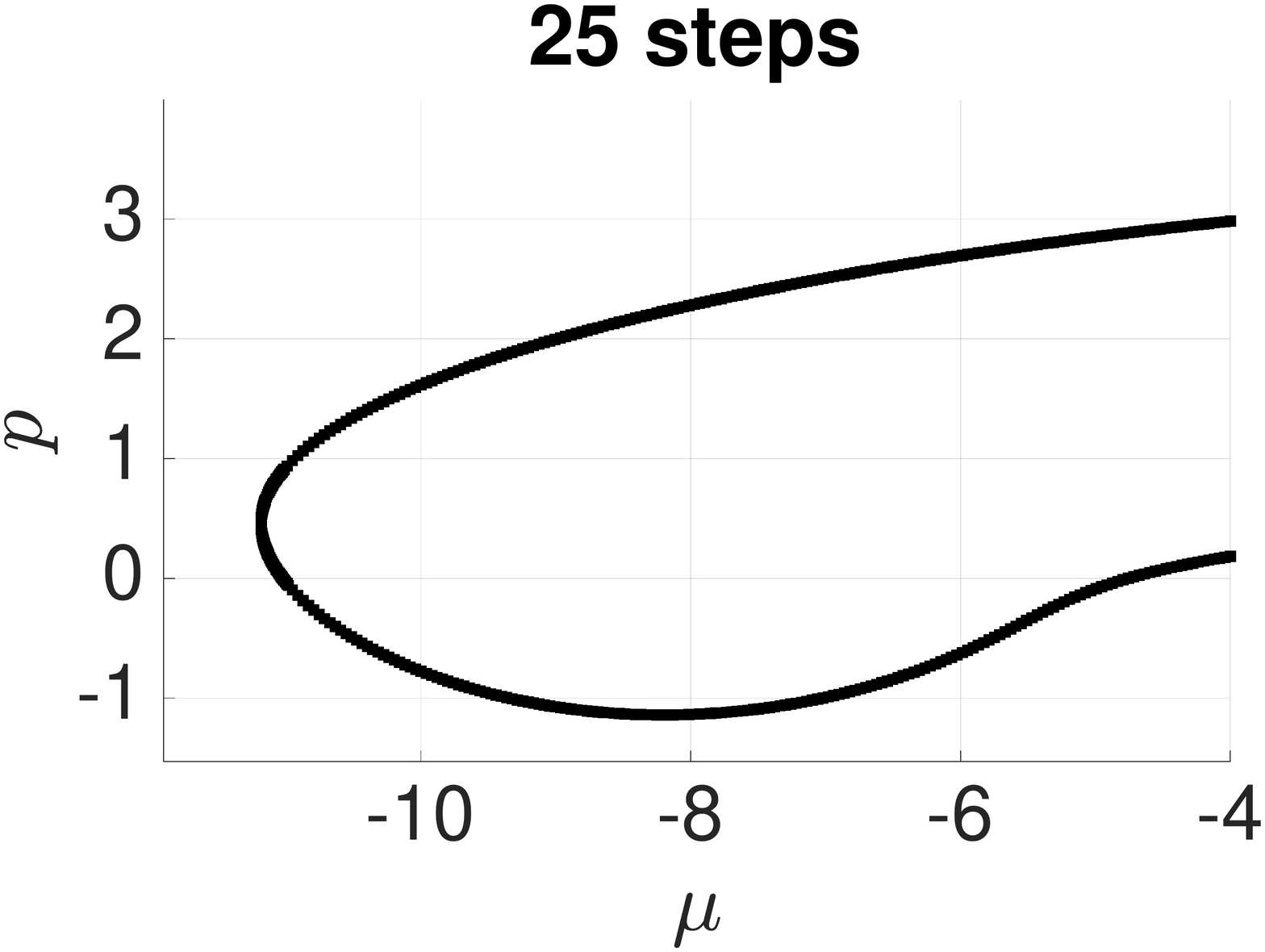}
\includegraphics[width=0.325\textwidth]{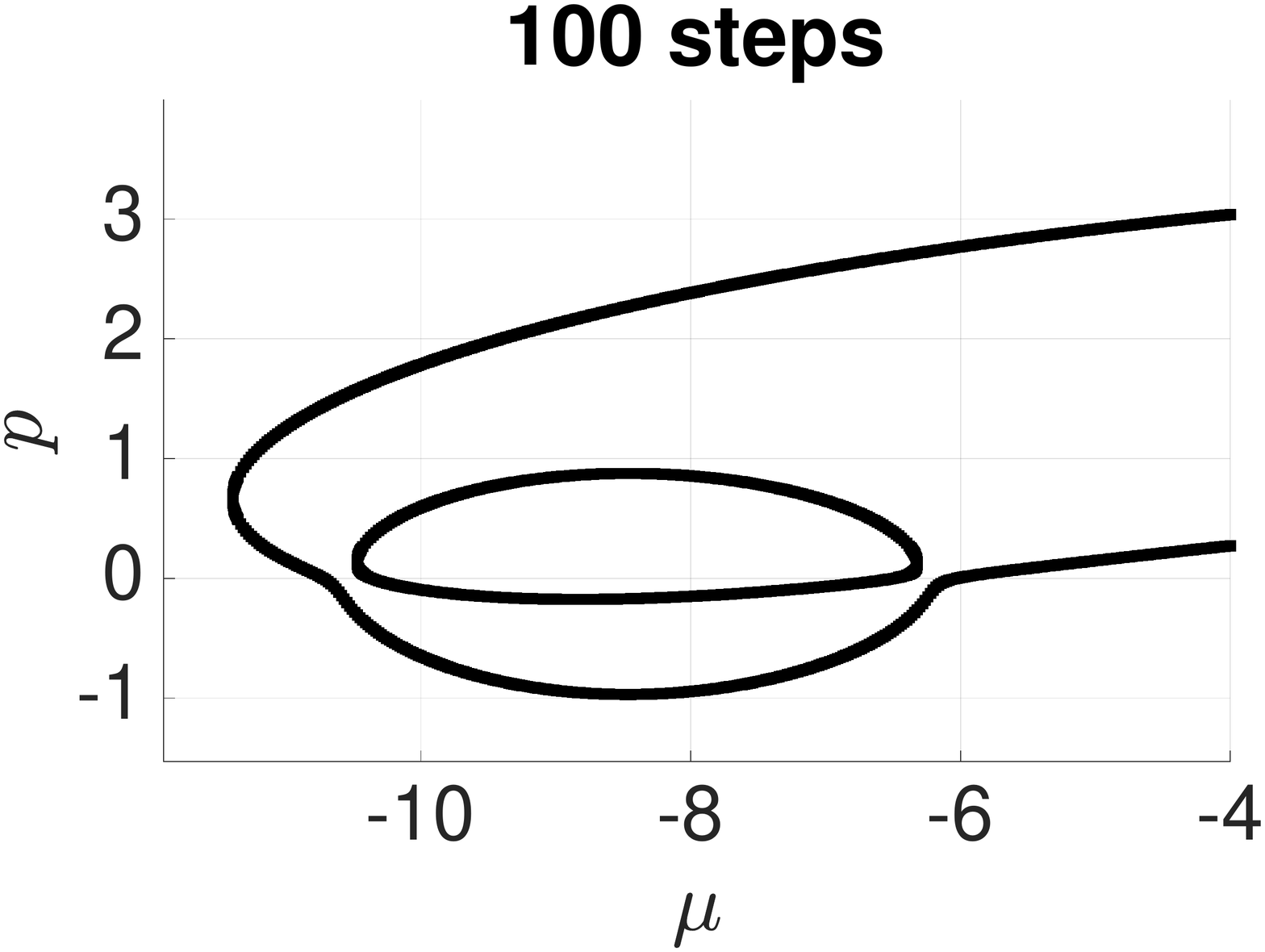}
\includegraphics[width=0.325\textwidth]{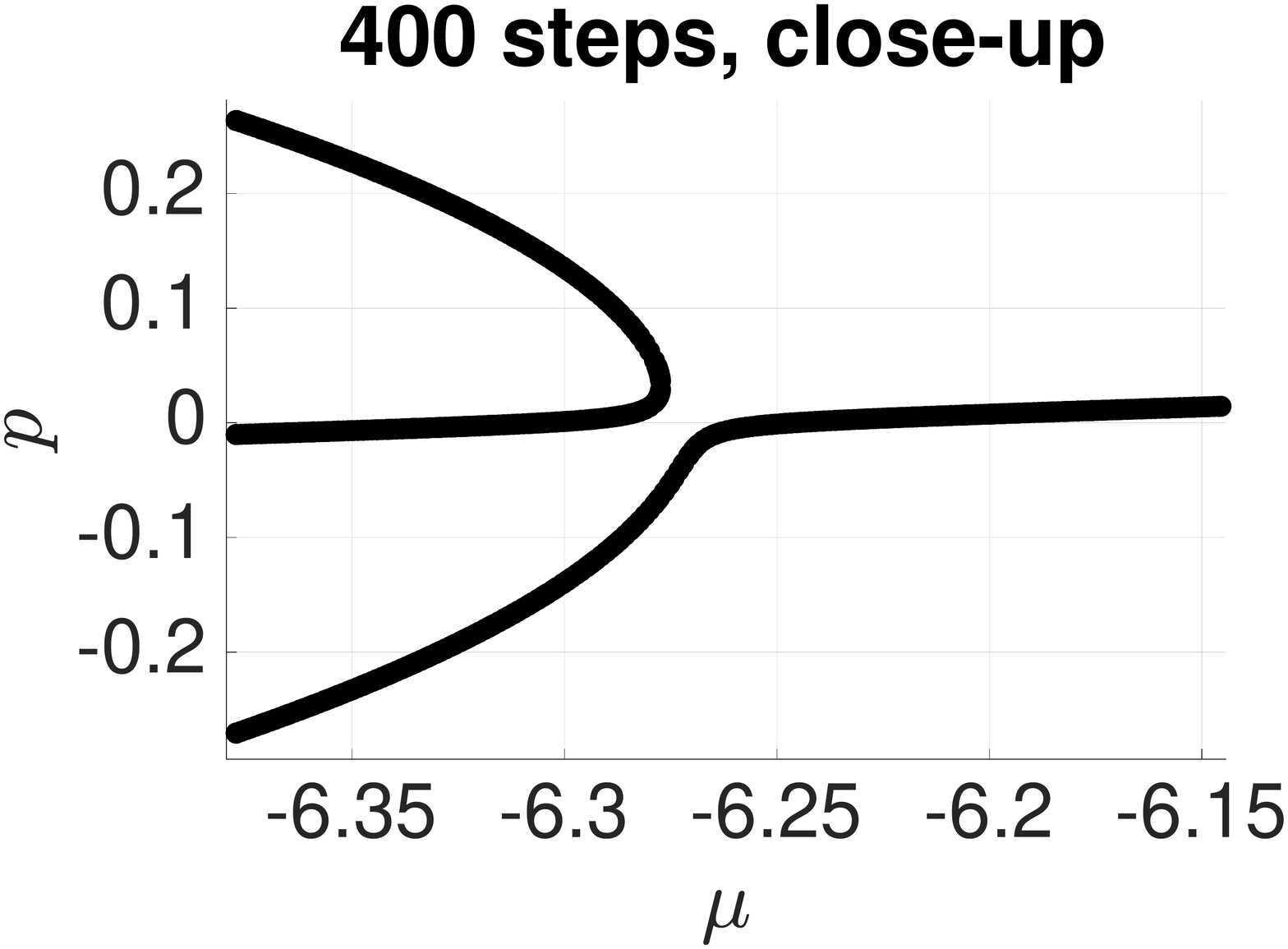}
\end{center}
\caption{Bifurcation diagrams for \eqref{eq:numEx} solved with RK2 using different number of time-steps.}\label{fig:pitchforkRK}
\end{figure}

{Figure \ref{fig:Lobatto} shows how the 4th order accurate, time-reversal symmetric, 3-stage Lobatto IIIA method captures the pitchfork bifurcation. The scheme is not symplectic. 
The symmetry properties of the method do not play a role as the considered Hamiltonian system is not time-reversal symmetric.
Comparing figure \ref{fig:pitchfork2d} with \ref{fig:Lobatto} we see that the St\"ormer-Verlet scheme beats the Lobatto IIIA method in terms of how well it preserves the shape of the periodic pitchfork bifurcation although it is of lower order.}

{Figure \ref{fig:MATLABbvp} shows the bifurcation diagram if the Hamiltonian boundary value problems are solved with MATLAB's build-in codes bvp4c and bvp5c (MATLAB R2016b). These are multi-purpose codes which are designed to solve general two-point boundary value problems for ODEs. 
The code bvp4c is based on the 4th-order, 3-stage Lobatto IIIA method while bvp5c uses the 4-stage Lobatto IIIA formula. Both codes re-mesh the time-grid if the solution does not meet tolerance criteria. The methods require an initial guess for a solution of the boundary value problem defined on a user supplied initial mesh \cite{KierzenkaBvP4c}.}
{
In our experiment we do the following (primitive) continuation method: We use initial guesses from the Lobatto IIIA experiment (figure \ref{fig:Lobatto}) at $\mu=-6.5$, let  bvp4c or bvp5c solve the boundary value problem and use the solution as a new initial guess for the boundary value problem at the next $\mu$ -value. The process is repeated for each branch. We leave the error tolerances at their default values. Allowing the methods to use up to $10^5$ mesh points in time, the codes run without issuing warnings. As $\mu$ varies, the codes adapt the time-meshes and we do not obtain consistent bifurcation diagrams. This is because the resulting diagram shows for each $\mu$ a snapshot of a bifurcation diagram of a different parameter-family of numerical flows. This illustrates that a $\mu$-dependent re-meshing strategy for the time-grid destroys the bifurcation diagram. 
}

{
Comparing figure \ref{fig:pitchfork2d} and \ref{fig:nonuniformmesh} shows that using a fixed, non-uniform mesh destroys the excellent behaviour of symplectic methods for capturing the periodic pitchfork bifurcation. This is in contrast to the bifurcations analysed in section \ref{sec:brokenBifur}, where using a non-uniform mesh does not change the behaviour of the integrator qualitatively.
The mesh used in the numerical example is the image of a uniform grid on the interval $[0,1]$ under the map $t \mapsto \tau \frac{\exp(5t)\sin(2.6 t)}{\exp(5)\sin(2.6)}$ with 226 and 905 grid points. Here, all step-sizes are smaller than in a uniform grid with 100 or 400 grid points, respectively. With a fixed, non-uniform mesh the St\"ormer-Verlet scheme still generates a symplectic flow map. However, it loses its energy conservation properties and behaves similar to RK2 when computing the bifurcation diagram of a periodic pitchfork (compare to figure \ref{fig:pitchforkRK}).
}

\begin{figure}
\begin{center}
\includegraphics[width=0.325\textwidth]{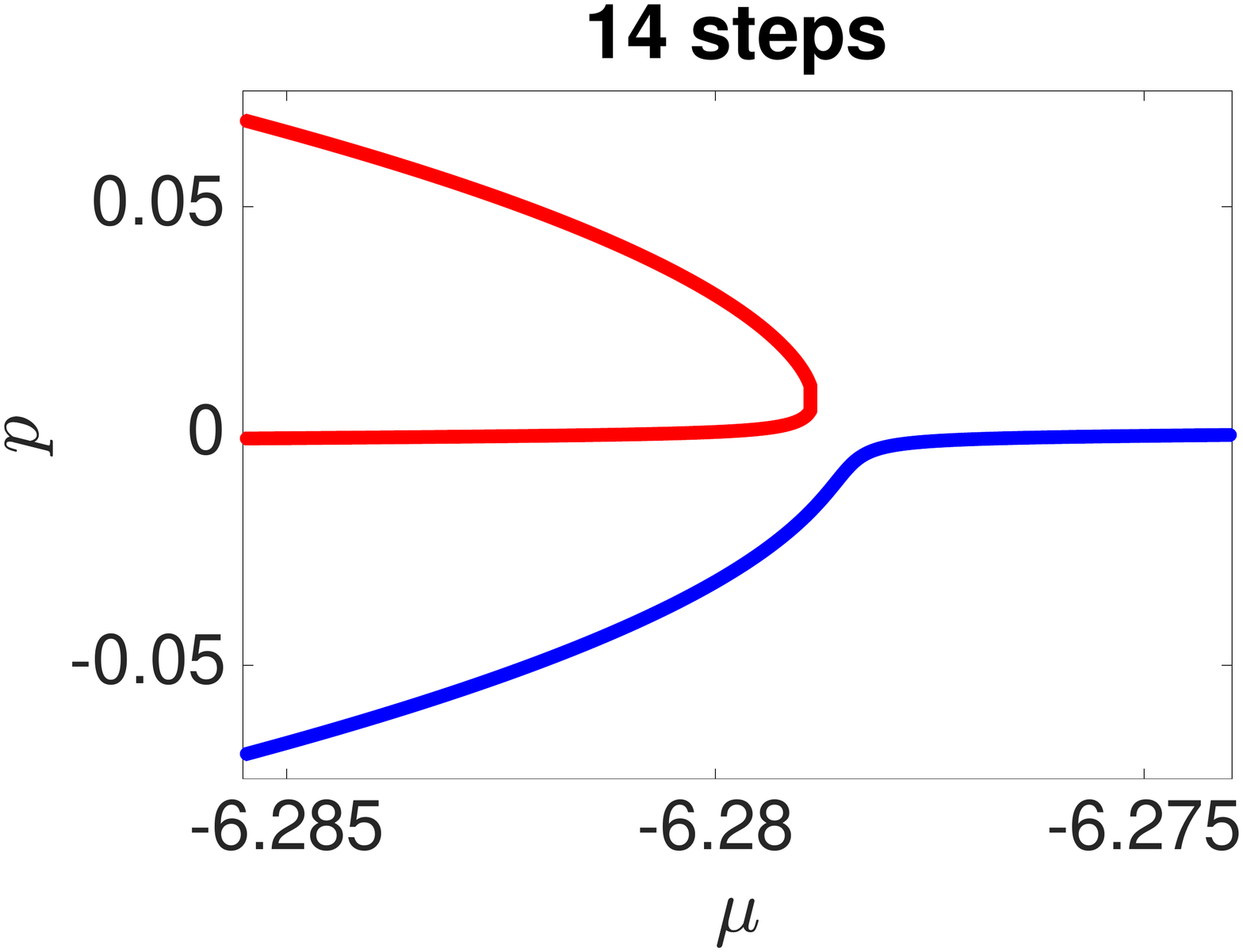}
\includegraphics[width=0.325\textwidth]{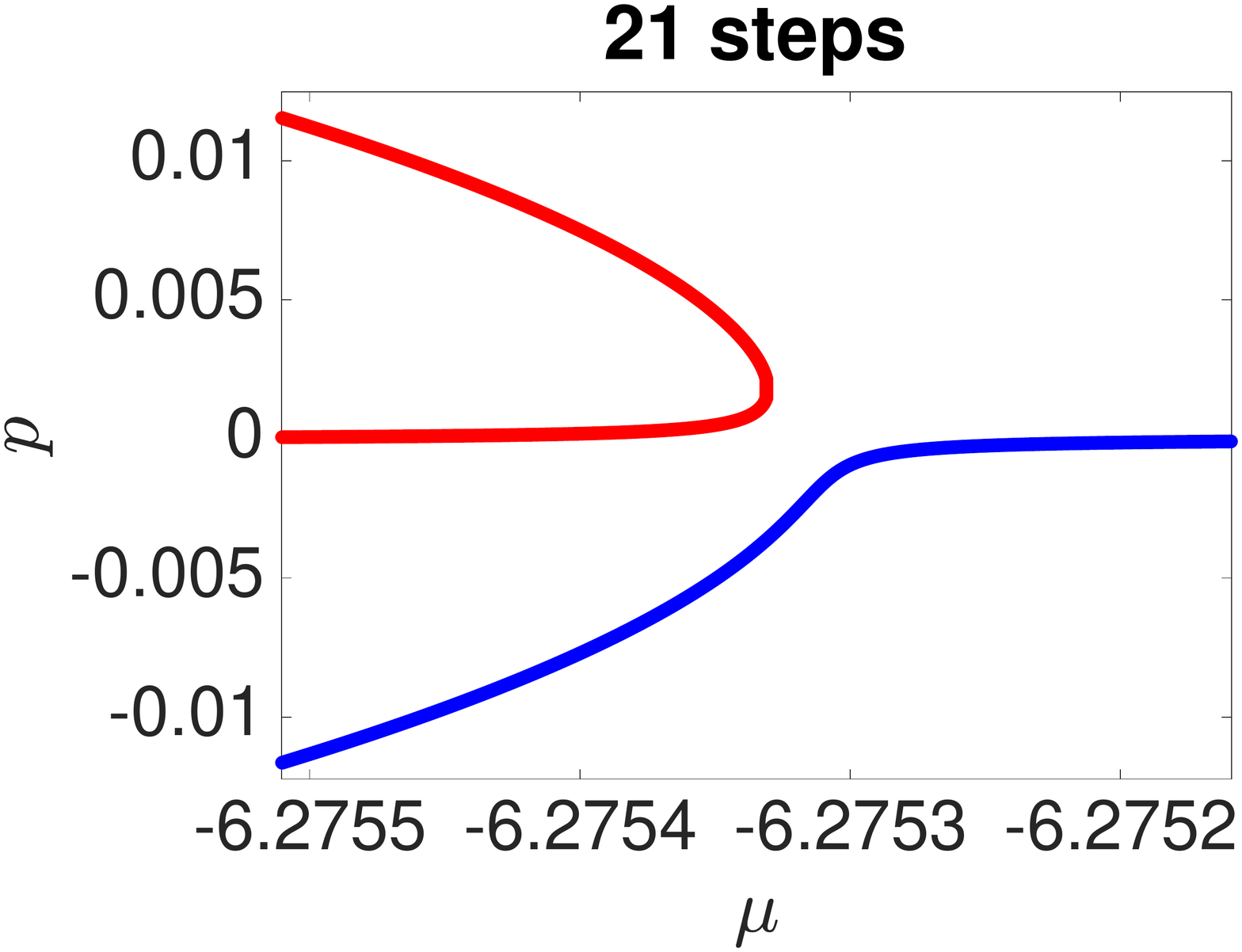}
\includegraphics[width=0.325\textwidth]{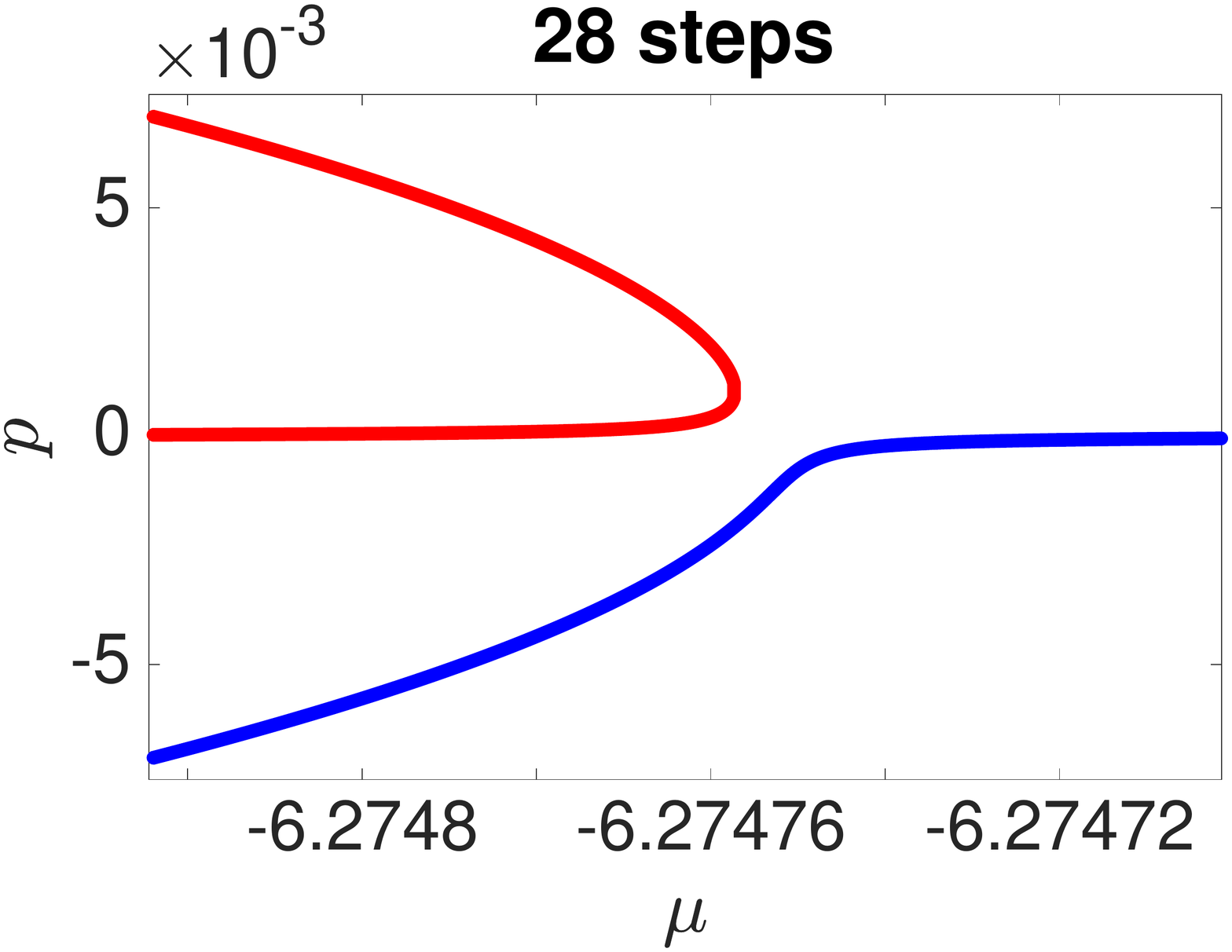}
\end{center}
\caption{
{
Bifurcation diagrams for \eqref{eq:numEx} solved with the 4th order Lobatto IIIA method and different number of time steps. The implicit equations arising in the method were solved up to round-off errors using Newton iterations.}}\label{fig:Lobatto}
\end{figure}

\begin{figure}
\begin{center}
\includegraphics[width=0.325\textwidth]{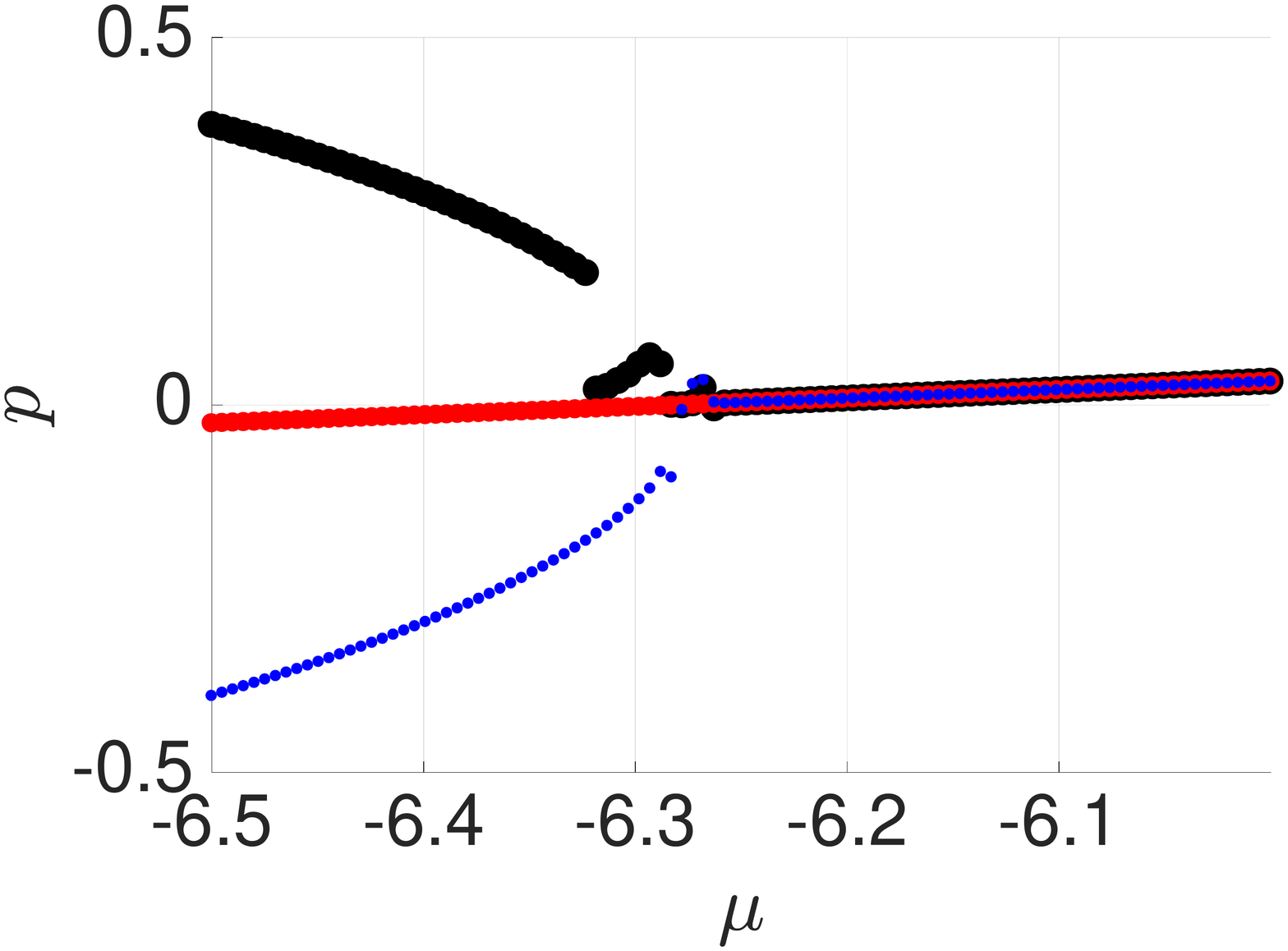}
\includegraphics[width=0.325\textwidth]{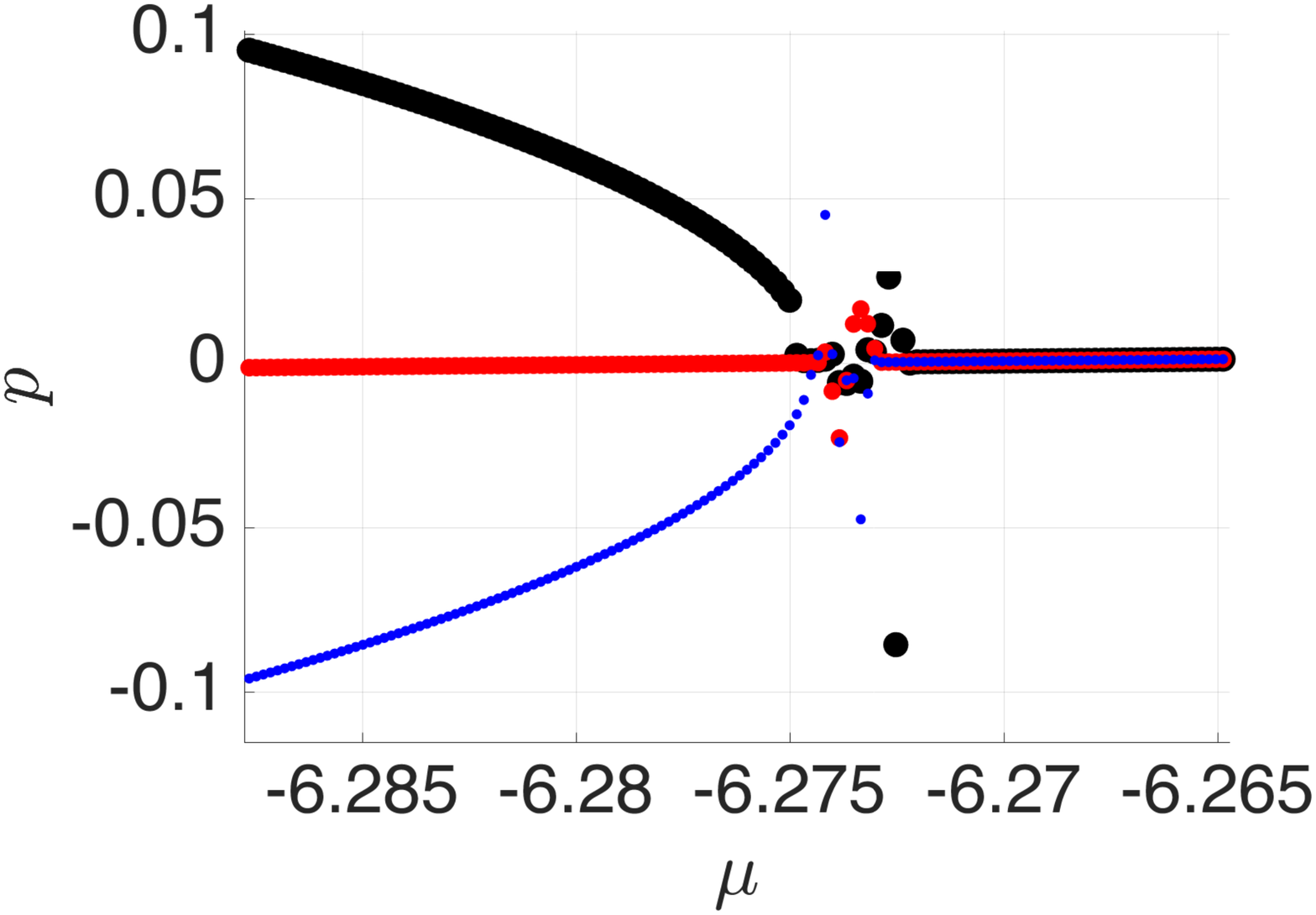}
\end{center}
\caption{
{Bifurcation diagrams for \eqref{eq:numEx} solved with MATLAB's bvp4c (left) and bvp5c (right), which use a re-meshing strategy destroying the bifurcation diagram.}}\label{fig:MATLABbvp}
\end{figure}

\begin{figure}
\begin{center}
\includegraphics[width=0.3\textwidth]{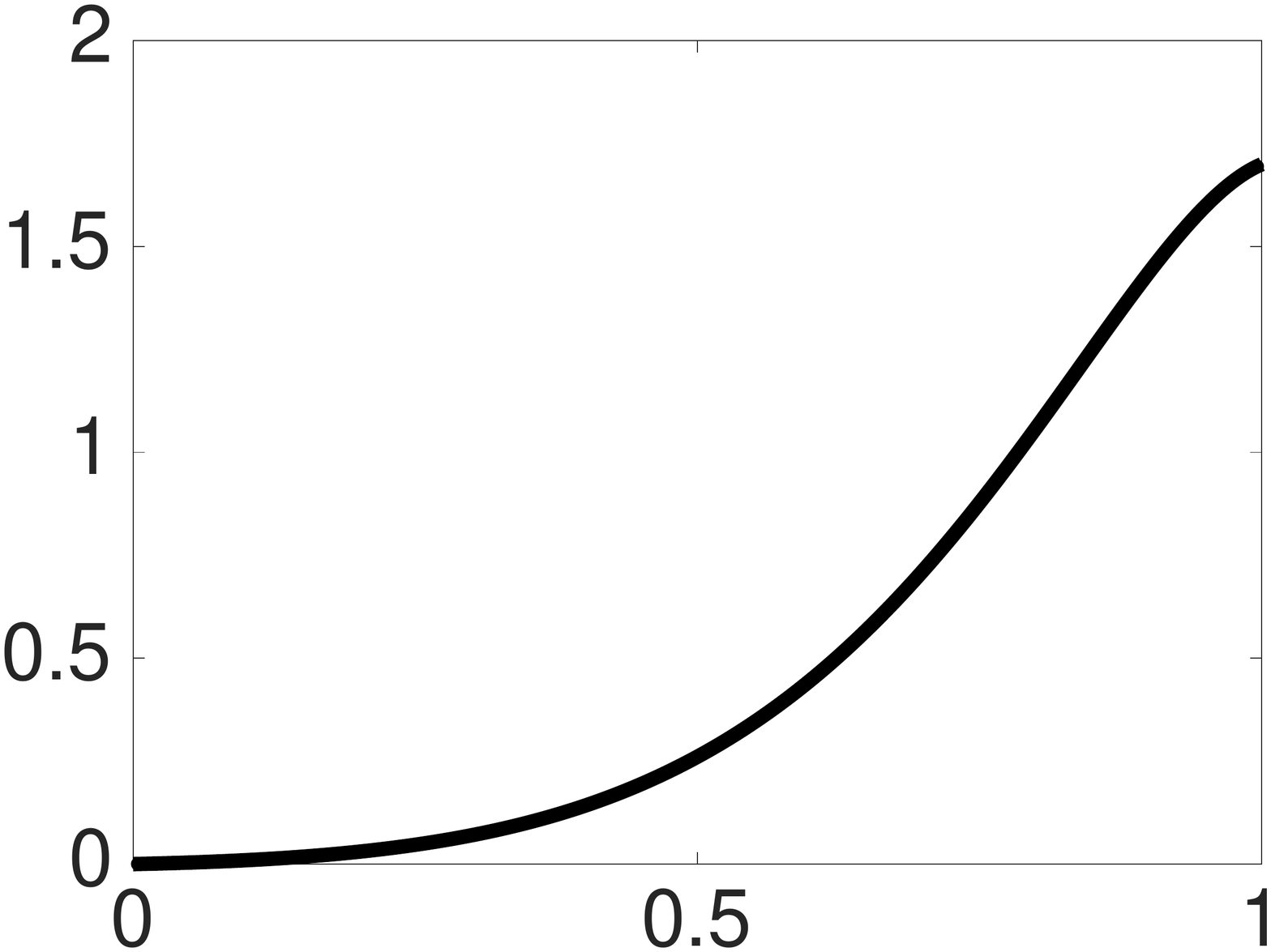}
\includegraphics[width=0.325\textwidth]{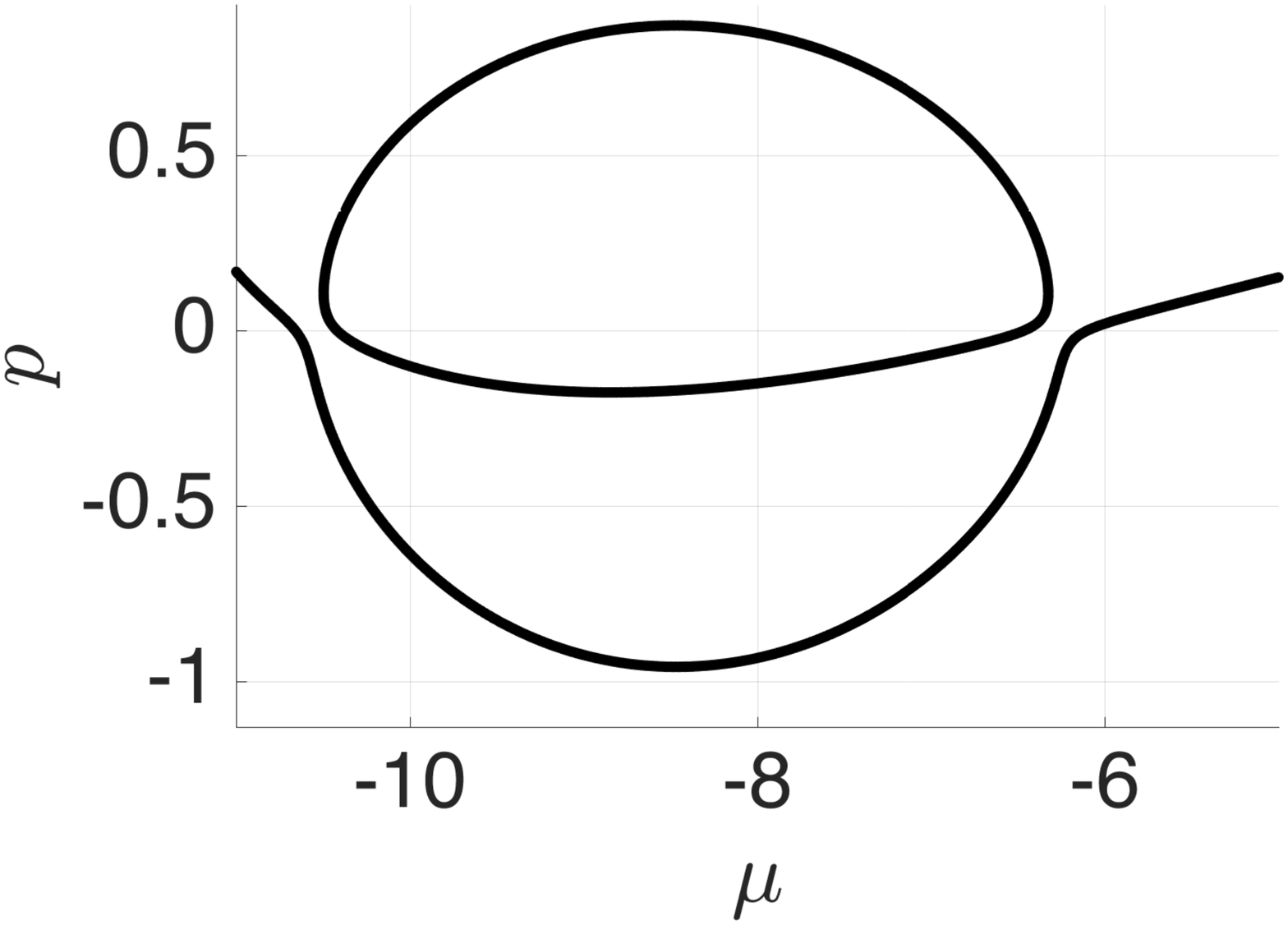}
\includegraphics[width=0.325\textwidth]{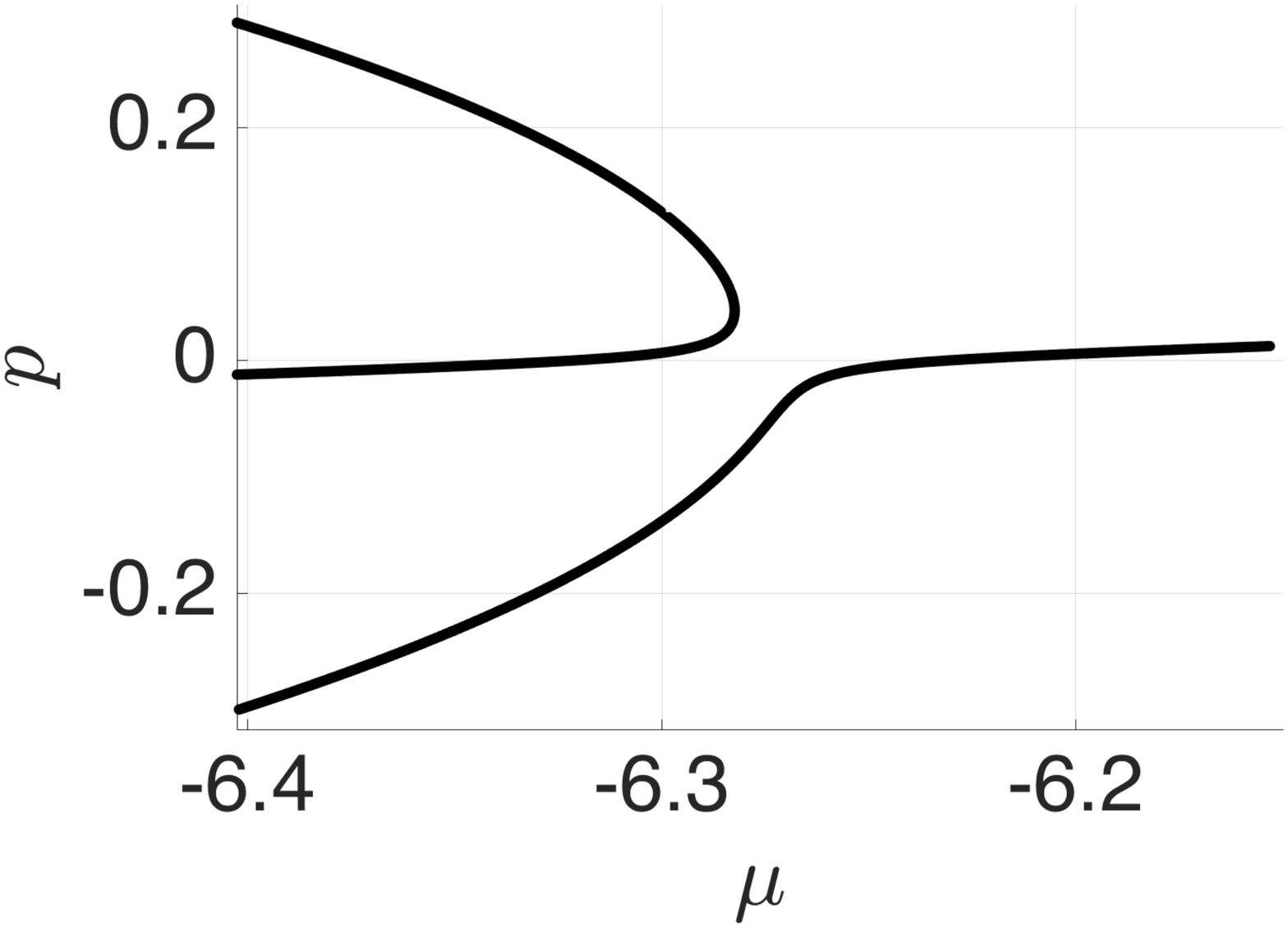}
\end{center}
\caption{
{Bifurcation diagrams for \eqref{eq:numEx} solved with the St\"ormer-Verlet method using a non-uniform mesh for the time-integration. The mesh is obtained by mapping a uniform mesh on the interval $[0,1]$ with the function whose graph is plotted to the left. The diagram in the centre corresponds to 226 mesh-points (finer than a uniform mesh with 100 points) and the diagram to the right corresponds to 905 mesh points (finer than a uniform mesh with 400 points). Here the St\"ormer-Verlet method behaves similar to RK2 (see figure \ref{fig:pitchforkRK}) and much worse than with a uniform grid (figure \ref{fig:pitchfork2d}). Also see remark \ref{rem:nonuniformmesh}.}
}\label{fig:nonuniformmesh}
\end{figure}


The St\"ormer-Verlet method preserves linear invariants \cite[Thm. IV 1.5]{GeomIntegration} and quadratic invariants of the form $Q(q,p)=q^t A p$ for a fixed matrix $A$ \cite[Thm. IV 2.3]{GeomIntegration}. Figure \ref{fig:PPcyclic1415newcoords} shows convergence of the numerical solution obtained with the St\"ormer-Verlet method to a pitchfork bifurcation in a 4-dimensional Hamiltonian system with a linear symmetry. The performance is much better than expected from the accuracy of the scheme. This can be compared to figure \ref{fig:captureintegrablepitchfork} showing a pitchfork bifurcation in a 4-dimensional Hamiltonian system with a non-affine linear symmetry. Here the periodic pitchfork bifurcation is captured only as well as expected from the accuracy of the integrator.

To which extent the completely integrable structure of a system is present in the numerical flow determines how well a pitchfork bifurcation is captured. Symplectic schemes have the advantage over non-symplectic integrators that they preserve a modified Hamiltonian exponentially well.
If, additionally, the other integrals of motions are also captured, e.g.\ because they are of a simple form or {arise from} a simple symmetry, then a symplectic method captures the periodic pitchfork bifurcations exponentially well \cite[section 4.5, Prop.7]{numericalPaper}. For a non-symplectic scheme for this to happen either \textit{all} integrals must be of a special form or be coming from simple symmetries for the method to capture these automatically or we must enforce their preservation (e.g.\ by a projection step) increasing computational costs. 
These observations can be extended to all bifurcations which make use of a completely integrable phase portrait.

\begin{figure}
\begin{center}
\includegraphics[width=0.325\textwidth]{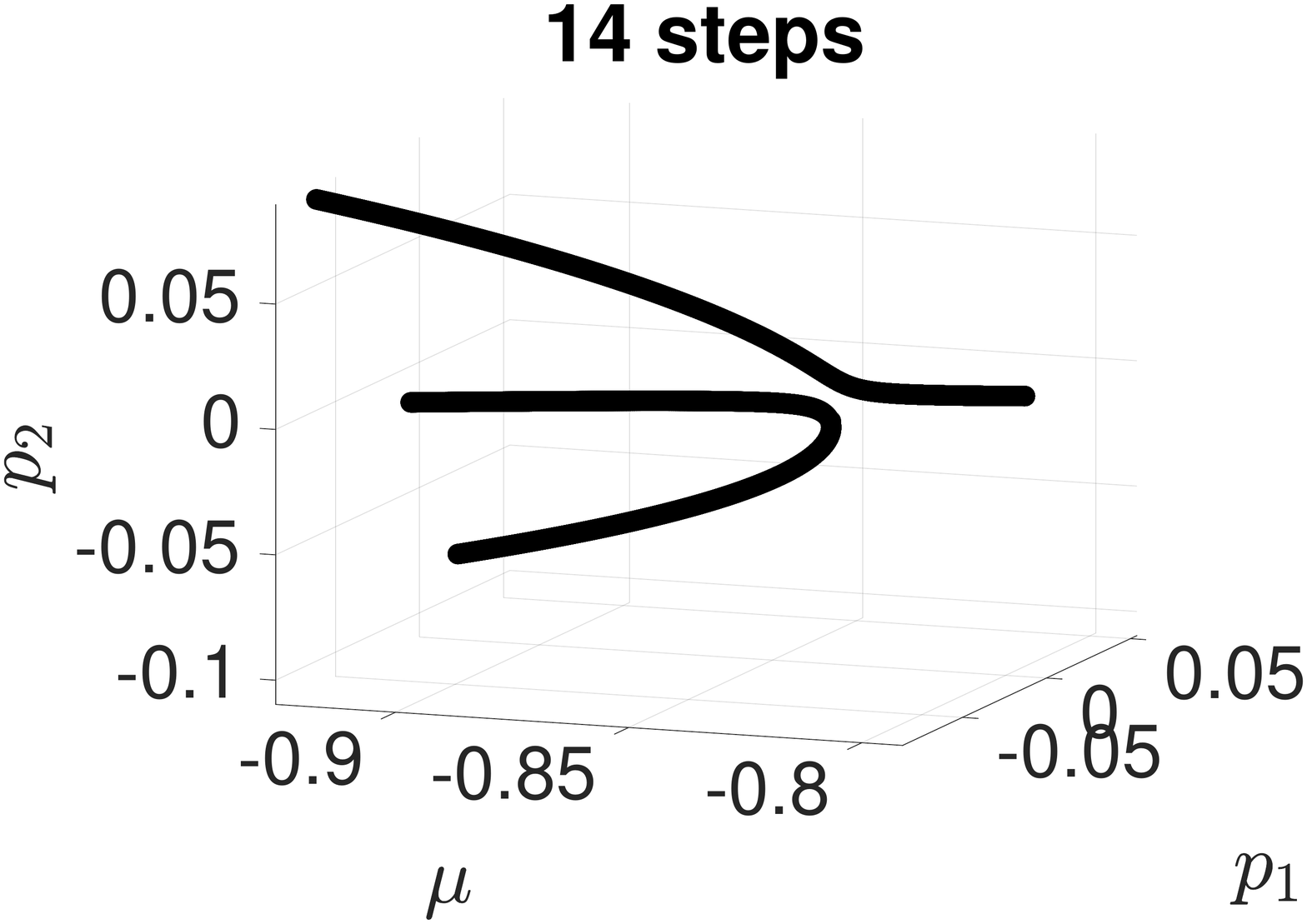}
\includegraphics[width=0.325\textwidth]{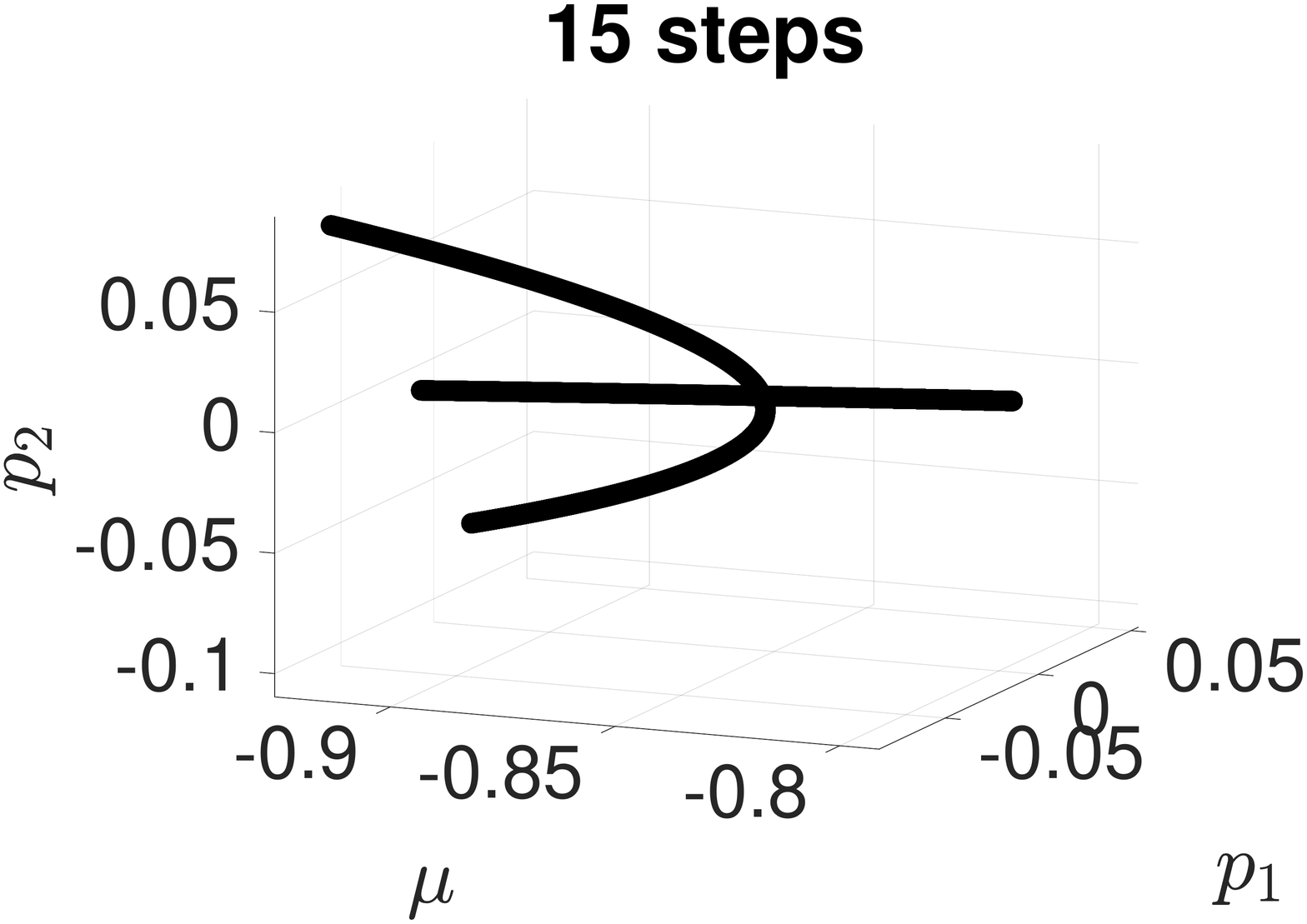}
\includegraphics[width=0.325\textwidth]{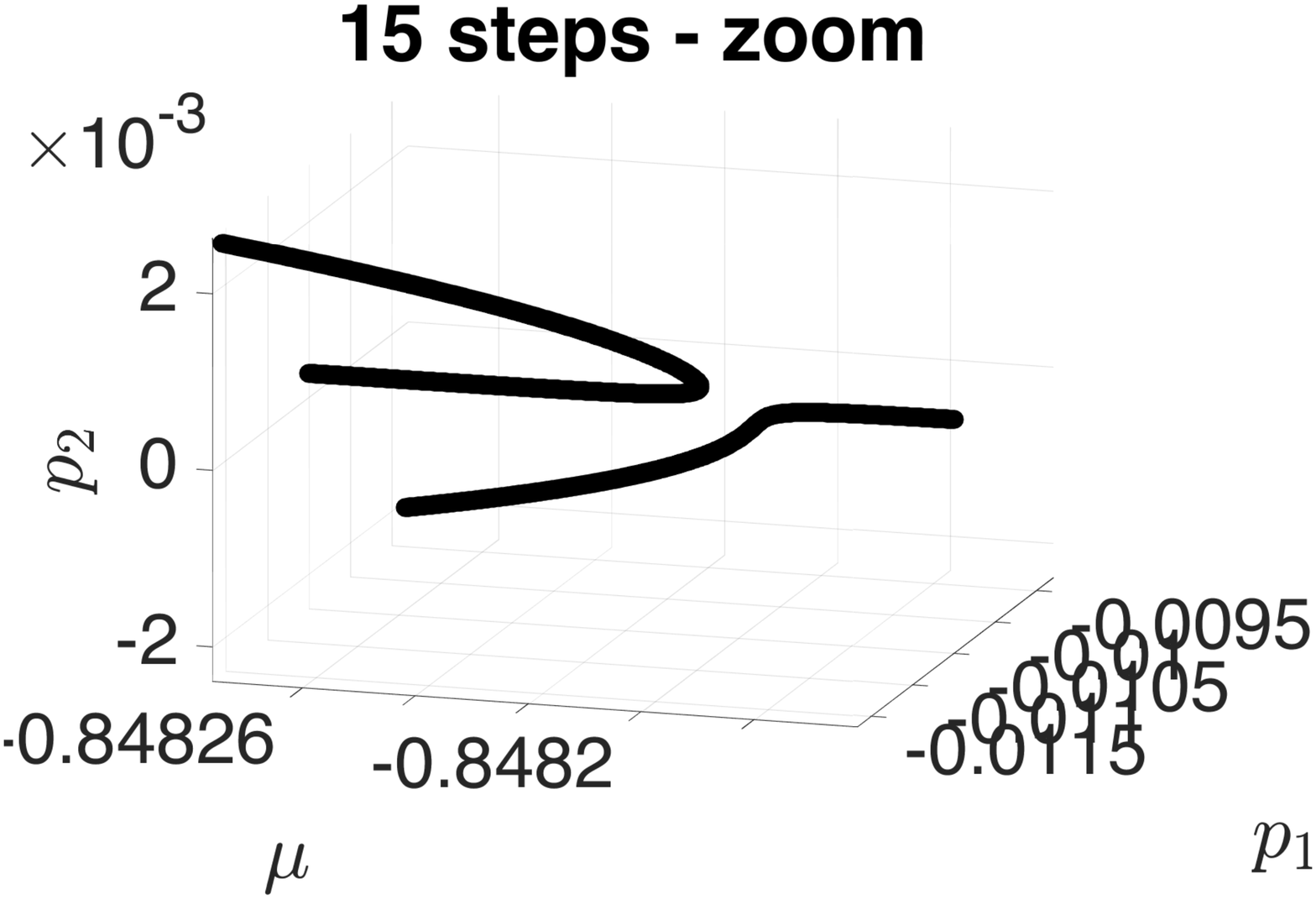}
\end{center}
\caption{Bifurcation diagrams obtained with the St\"ormer-Verlet method in a Hamiltonian system with an additional linear invariant. We consider ~\usebox{\smlmata} for
$\overline H_\mu(\overline q^1, \overline q^2,\overline p_1,\overline p_2)= (\overline q^1)^3+\mu \overline q^1 + \overline p_1 \overline p_2 + \overline p_1^2 + \frac 1 {10}(\overline p_1^3+\overline p_2^3)$ after a linear transformation with the co-tangent lifted action of the linear change of variables \usebox{\smlmatb}. While the break is clearly visible when 14 steps are used, it can only be spotted in a close-up when 15 steps are used.}\label{fig:PPcyclic1415newcoords}
\end{figure}

\begin{figure}
\begin{center}
\includegraphics[width=0.32\textwidth]{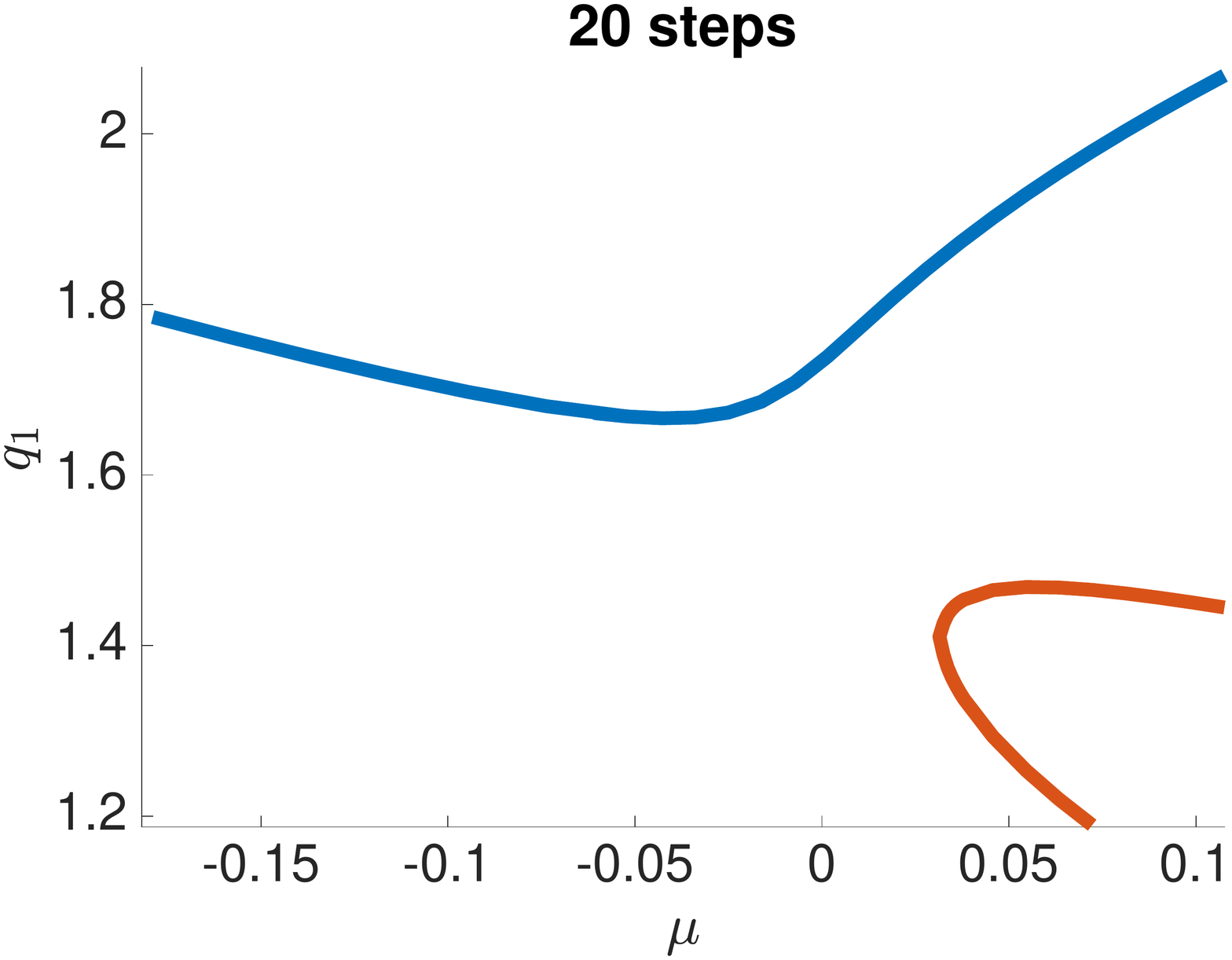}
\includegraphics[width=0.32\textwidth]{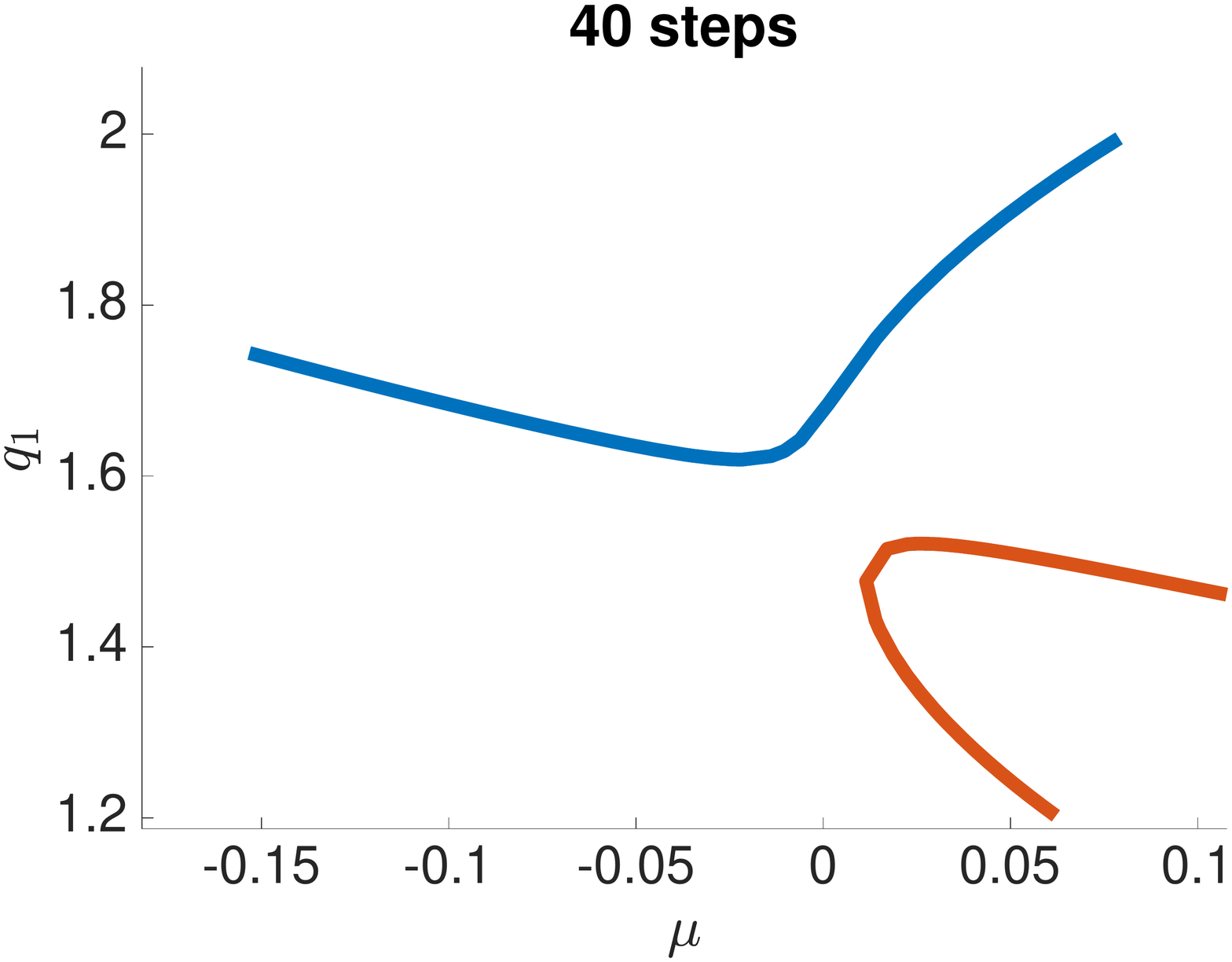}
\includegraphics[width=0.32\textwidth]{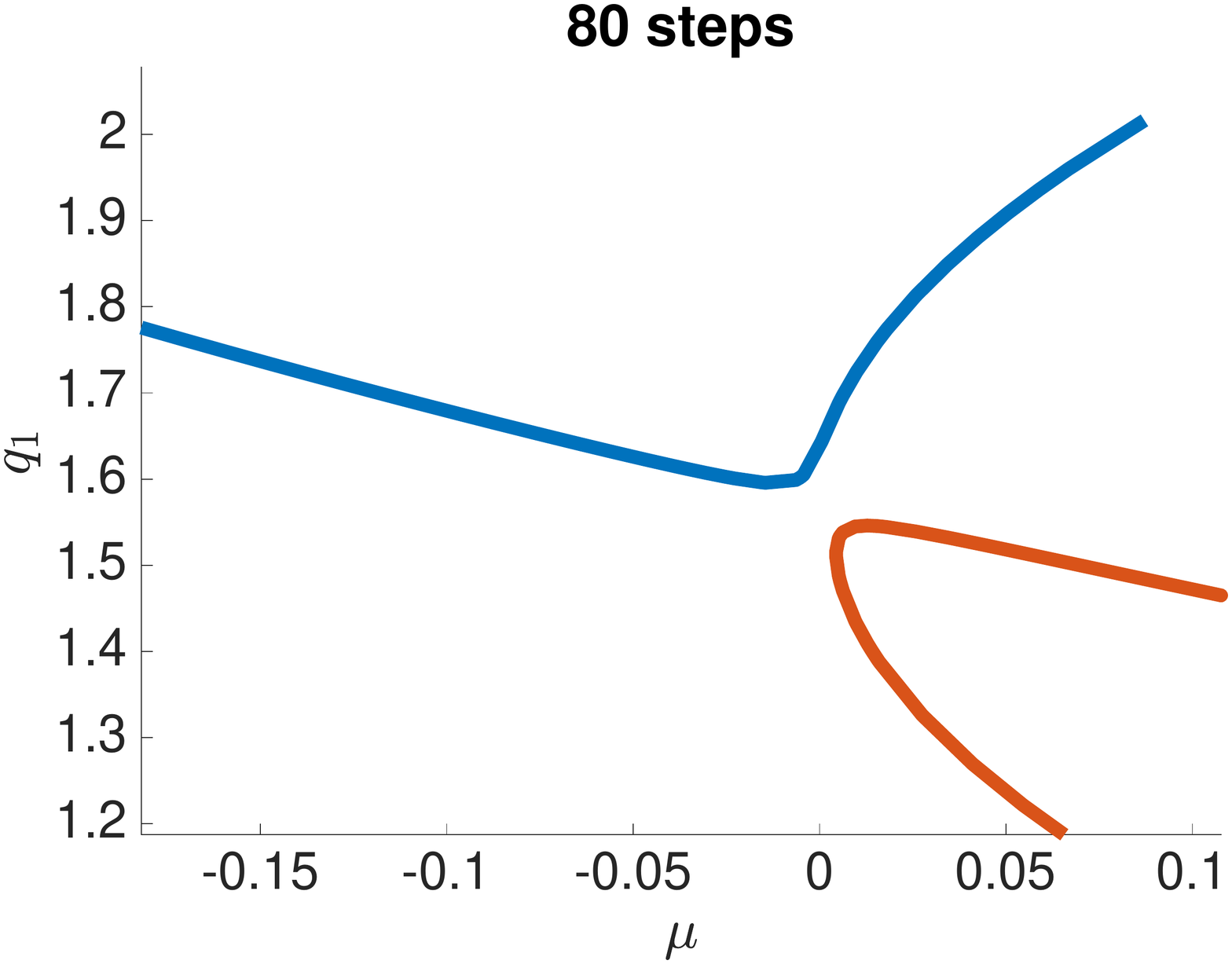}
\end{center}
\caption{Bifurcation diagrams obtained using the St\"ormer-Verlet method, projected along the $q_2$-axis for the problem \usebox{\smlmatc} for $\overline H_\mu (\overline q, \overline p) = \overline p_1^3 + \mu \overline p_1 + \overline p_2^2$ on $T^\ast(S^1\times S^1)$ transformed with the co-tangent lifted action of $\overline q_1 = q_1+0.1 \cos(q_2), \overline q_2 = q_2+0.1 \cos(q_1)$. The integrals $\overline p_1$ and $\overline p_2$ are nonlinear functions of $(q,p)$ and are \textit{not} preserved.}\label{fig:captureintegrablepitchfork}
\end{figure}

\section{Further remarks}\label{sec:Remarks}

\begin{remark}[non-uniform meshes]\label{rem:nonuniformmesh}
The bifurcations considered in section \ref{sec:brokenBifur} are connected to the symplecticity of the Hamiltonian flow but not to the preservation of the Hamiltonian. We can, therefore, use a symplectic integrator together with a fixed but not necessarily uniform grid for the integration of Hamiltonian ODEs to preserve the bifurcations. In contrast, the periodic pitchfork bifurcation analysed in section \ref{sec:capturePitchfork} is related to the energy preservation of the Hamiltonian flow. To make use of the excellent energy behaviour of symplectic integrators using a uniform mesh is essential.
\end{remark}


\begin{remark}[Conjugate symplectic methods]\label{ref:conjsympl}
{
Numerical methods are called {\em conjugate} if their numerical flows $\Phi_h$ and $\Psi_h$ are related by a change of coordinates $\chi_h$ such that $\Phi_h = \chi_h \circ \Psi_h \circ \chi_h$.
An example is the trapezoidal rule which is conjugate to the implicit midpoint rule \cite[VI.8]{GeomIntegration}. 
Numerical flows obtained using a conjugate symplectic method preserve a nearby symplectic form and share the excellent energy behaviour with flows obtained by a symplectic method because the flows are conjugate. To capture periodic pitchfork bifurcations, they are just as good as symplectic methods.
In contrast, this is {\em not} true for the higher gradient-zero bifurcations considered in section \ref{sec:brokenBifur} because (unless in a degenerate situation) the boundary condition will not be Lagrangian w.r.t.\ the modified symplectic form such that the bifurcations in the numerical system are broken up to the order of accuracy.
Some methods are conjugate to a symplectic method up to some (high) order. An example is Labatto IIIA, which is conjugate symplectic up to order 6 \cite{ConjSymplBseriesHairer2012}.
If a $k$-order method is conjugate symplectic up to a high order $r>k$ then it will behave as good as an $r$-order scheme in resolving the periodic pitchfork bifurcation but will show order $k$ broken $D$-series bifurcations in generic Hamiltonian boundary value problems.
}
\end{remark}

\begin{acknowledgements}
We thank Peter Donelan, Bernd Krauskopf, Hinke Osinga and Gemma Mason for useful discussions. This research was supported by the Marsden Fund of the Royal Society Te Ap\={a}rangi.
\end{acknowledgements}

\bibliographystyle{spmpsci}      
\bibliography{resources}   

%
%

\end{document}